\documentclass[11pt,
]{amsart}
\pdfoutput=1
\usepackage[english]{babel}
\usepackage{amsthm,amsmath,amssymb,amsfonts}
\usepackage[colorlinks=true,urlcolor=blue,linkcolor=blue,citecolor=blue,backref=page]{hyperref}
\usepackage[alphabetic,backrefs]{amsrefs}
\usepackage{mathtools} 
\usepackage{graphicx}
\usepackage{tikz}
\usepackage{setspace}
\usepackage{xcolor}
\usepackage{color}
\newcommand{\brown}[1]{{\color{brown}#1}}
\newcommand{\blue}[1]{{\color{blue}#1}}
\newcommand{\gray}[1]{{\color{gray}#1}}

\newtheorem{theorem}{Theorem}[section]
\newtheorem{corollary}[theorem]{Corollary}
\newtheorem{lemma}[theorem]{Lemma}

\newtheorem{problem}[theorem]{Problem}
\newtheorem{conjecture}[theorem]{Conjecture}
\newtheorem{proposition}[theorem]{Proposition}
\theoremstyle{definition}
\newtheorem{definition}[theorem]{Definition}
\newtheorem{remark}[theorem]{Remark}
\newtheorem{example}[theorem]{Example}

\newcommand{\ds}{\displaystyle}

\newcommand{\Z}{\ensuremath{\mathbb{Z}}}

\newcommand{\rank}{\ensuremath{\operatorname{rank}}}
\newcommand{\im}{\ensuremath{\operatorname{im}}}

\newcommand{\arr}{\ensuremath{\operatorname{arr}}}

\newcommand{\Addresses}{{
  \bigskip
  \footnotesize
  \textsc{
Department of Mathematics \& Statistics, McMaster University\\
Hamilton, Ontario,
Canada L8S 4K1
  }\par\nopagebreak
      \textit{E-mail address}: \texttt{chenj293@mcmaster.ca}
  }}
\title[FlatKnotInfo: the first 1.24 million flat knots]{FlatKnotInfo:\\ the first 1.24 million flat knots}
\author{Jie Chen}
\date{}
\begin{document}
\subjclass[2020]{57K10; 57K12}
\keywords{flat knot, knot table, flat knot invariant}
\begin{abstract}
    We use matchings on Lyndon words to classify flat knots up to 8 crossings. Using flat knots invariants such as the based matrix, the $\phi$-invariant, the flat arrow polynomial, and the flat Jones-Krushkal polynomial, we distinguish all flat knots up to 7 crossings except for five pairs. 
    Among the many flat knots considered, we find examples that are: (i) algebraically slice but not slice;
    (ii) almost classical (null-homologous) but not slice; 
    (iii) nontrivial but with trivial (primitive) based matrix.
    
    The classification data has been curated and is available on FlatKnotInfo, which is an interactive searchable website listing flat knots up to 8 crossings and their invariants. It also provides access to algebraic and diagrammatic information on these knots and is designed to enable users to discover patterns and formulate conjectures on their own. 

    \end{abstract}
\maketitle

\section{Introduction}

Knots have been part of human culture since ancient times. For instance, interlaced patterns of knots and links appear as early as  the $3^{\rm rd}$ and $4^{\rm th}$ centuries, and they are featured prominently in early Celtic art and in the famous Book of Kells. Despite this, the study of knots as mathematical objects did not begin until around the $18^{\rm th}$ century. 

One of the central problems in knot theory is the classification problem, whose solution can be divided into two steps. The first step is to generate a \textit{complete} list of representatives, and the second is to remove duplicates from the list until it includes exactly one representative for each knot type. The first is a problem of generation, and its solution requires constructive methods. The second is a problem of separation, and its solution requires obstructive methods. Ideally, an algorithm is developed to generate representatives and knot invariants are applied to distinguish them as distinct knot types.

For classical knots, the classification problem can be traced back to the $19^{\rm th}$ century, when Tait and Little constructed the first table of alternating classical knots up to 11 crossings. Since then, the knot tables were extended to knots with 11 crossings by Conway \cite{Conway}, then to 12 crossings by Rolfsen \cite{Rolfsen}, then to 17 crossings by Hoste and Thistlethwaite \cite{HTW98}, and most recently to 19 crossings by Burton \cite{Burton}.
Along the way, mathematicians have developed more efficient algorithms to generate representatives and devised many new invariants of knots and links.   
 
The large data sets of knots requires a compact notation for encoding diagrams, and two common methods are Dowker codes and Gauss codes, see \cite{knotinfo}. Both record the crossing information, and although every knot diagram determines a unique Gauss code, not all Gauss codes correspond to actual knot diagrams, in fact many are non-planar. To address this shortcoming, Kauffman introduced a new type of crossing called a \emph{virtual crossing}, and this led to the development of \emph{virtual knot theory} \cite{kauffman99}.
As a generalization of classical knots, virtual knots represent smooth embeddings of loops in a thickened closed oriented surface up to isotopy and stabilization. 
Many results and techniques for classical knots have been adapted and extended to virtual knots, including 
the Jones polynomial and Khovanov homology (cf.~\cites{kauffman99, GPV00,manturovbook}).  

However, virtual knot theory exhibits new and unexpected behavior. 
For example, Kronheimer and Mrowka \cite{KM11} proved that Khovanov homology detects the unknot, but this is no longer true among virtual knots. Indeed,  there are examples of nontrivial virtual knots with trivial Khovanov homology.
In addition, the operation of connected sum of virtual knots exhibits strange new behavior. It is not well-defined as an operation on virtual knots, rather it depends on the diagrams used as well as the connection points on the diagram. The Kishino knot is a striking example; it is a nontrivial virtual knot obtained as the connected sum of two trivial virtual knot diagrams!
 
For classical knots, it is well-known that any knot diagram can be transformed into a trivial knot with a sequence of crossing changes.
This property fails for virtual knots, and the Kishino knot 
is one that cannot be unknotted by crossing changes.
In short, crossing change is not an unknotting operation for virtual knots.

\begin{figure}[ht]
    \centering
    \includegraphics[scale=0.9]{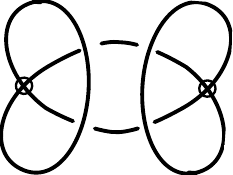}
    \hspace{32pt}
    \includegraphics[scale=0.9]{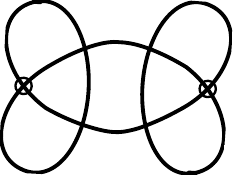}
    \caption[Kishino knot]{The virtual and  flat Kishino knot}
    \label{fig:kishino}
\end{figure}

Flat knots arise as the quotient of virtual knots modulo crossing change, 
 thus they represent the obstruction to a virtual knot being unknottable by crossing changes. 
Flat knots can be represented diagrammatically as flattened virtual knot diagrams, namely by replacing every classical crossing with a self-intersection. There is a well-defined surjection, called the \emph{shadow projection}, from virtual knots to flat knots, and all classical knots are mapped to the trivial flat knot under this map.

Many of the tools, methods, and invariants developed to study virtual knots can be adapted and applied to flat knots.
Flat knots represent immersed loops in closed oriented surfaces up to homotopy and stabilization. 
Although the Jones polynomial is trivial on flat knots, many of the other skein-based  polynomial invariants of virtual knots map are nontrivial on flat knots. For instance, the arrow polynomial and the Jones-Krushkal polynomial give strong invariants of flat knots and are useful in distinguishing them.

In \cite{turaev04}, Turaev developed an algorithm for classifying flat knots, and this approach was implemented by Gibson \cite{gibson}. He represented flat knots as nanowords, and used the $u$-polynomial, $\phi$-invariants, and the 2-parity projection to distinguish the flat knot types. This approach works for flat knots with up to four crossings, but the invariants are not sufficiently powerful to distinguish flat knots with five or more crossings.

In this paper, we give a tabulation of the first 1,289,741 flat knots. This is achieved by representing flat knots by Lyndon words and using a much larger set of invariants to distinguish them. Our tabulation includes flat knots with up to 8 crossings, and the results are available through an interactive website FlatKnotInfo \cite{flatknotinfo}, which lists the flat knots along with their invariants. 

These methods completely distinguish flat knots up to 6 crossings, and they work for flat knots with 7 crossings, leaving only 5 pairs of ambiguities.
We calculate invariants such as Reidemeister-3 orbits, mirror images and symmetry type, minimal genus, checkerboard colorability, and almost classicality.
We also calculate the based matrix and primitive based matrix, as well as invariants derived from the based matrix, such as the $\phi$-invariants and inner and outer characteristic polynomials.
We also calculate a number of different polynomial invariants, such as the $u$-polynomial, the flat Jones-Krushkal polynomial, the flat arrow polynomial and the 2-strand cabled flat arrow polynomial. 

In tabulating a large number of flat knots, we observed many interesting patterns and discovered new phenomena. Based on the empirical data, we initially formulated the following three hypotheses:
\begin{enumerate}
\item[(A)] For flat knots, algebraically slice implies slice.
\item[(B)] For flat knots, the primitive based matrix detects the unknot.
\item[(C)] Every almost classical flat knot is algebraically slice.
\end{enumerate}
While these hypotheses had been supported by low-crossing number flat knots, we found counterexamples to each one. For instance, Example~\ref{eg:6464} presents a flat knot that is algebraically slice but not slice, disproving (A).
Figure~\ref{fig:fkac828} shows a nontrivial flat knot with trivial primitive based matrix, disproving (B).
The flat knot in Figure~\ref{fig:fk7nonslice} is another counterexample; it is  nontrivial and not even slice but it has trivial primitive based matrix.
Statement (C) was the most difficult to disprove, in fact it is true for flat knots up to 10 crossings. However, there are counterexamples among the 11-crossing flat knots. In particular, the flat knots in Figure~\ref{fig:ac11} are almost classical but not algebraically slice, and they give counterexamples to (C).

Other computational results are available online on FlatKnotInfo \cite{flatknotinfo}, where users can access pre-calculated invariants and search for flat knots. FlatKnotInfo was inspired by the famous knot theory website KnotInfo \cite{knotinfo}, and it also owes an intellectual debt to the incredibly useful online table of virtual knots \cite{Green}. Each flat knot has a separate page featuring its minimal Gauss diagrams and invariants, along with cross-references to other flat knots with same invariant values and to virtual knots mapping to it under shadow projection. FlatKnotInfo also provides information on sliceness and algebraic sliceness for all flat knots up to 7 crossings with nine exceptions.

FlatKnotInfo  \cite{flatknotinfo} was an indispensible tool in carrying out the research presented here, and it served to guide and inform our investigations. Indeed one goal of this paper is to showcase the discoveries made in developing FlatKnotInfo, which grew out of an effort to classify the largest possible dataset of flat knots (cf. \cite{chen-thesis}). FlatKnotInfo  \cite{flatknotinfo} provides the computational backbone for many of the invariant calculations that appear below, including based matrices, arrow polynomials, Jones-Krushkal polynomials. For readers interested in more detail on the computational aspects, including algorithms and Python code, please refer to \cite{chen-thesis}.

The remainder of this paper is organized as follows.
In Section~2, we review preliminary materials.
In Section~3, we briefly review the theoretical basis for tabulating flat knots combinatorically, followed by the algorithm used in our tabulation.
In Section~4, we introduce the based matrix and several invariants derived from it. These invariants are used in Section~5, where we show the difference between sliceness, algebraic sliceness and almost-classicality of flat knots.
 In Sections~6 and 7, we introduce the flat arrow polynomial and the flat Jones-Krushkal polynomial, respectively.
In Section~8, we make a few concluding remarks.

\section{Basic notions}
In this section, we introduce flat knots which are described interchangeably in terms of flat knot diagrams, flat Gauss diagrams and immersions in oriented closed surfaces.

\begin{definition}
A \emph{flat link diagram} is a $4$-valent planar graph, 
where each vertex is now allowed to be either a \emph{flat crossing} or virtual crossing
as in Figure~\ref{fig:fcr}.
\end{definition}

\begin{figure}[ht]
    \centering
    \includegraphics[scale=1.2]{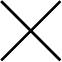}
    \hspace{26pt}
    \includegraphics[scale=1.2]{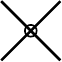}
    \caption[Flat and virtual crossings]{A flat crossing (left) and virtual crossing (right) }
    \label{fig:fcr}
\end{figure}

\begin{definition}
    Two flat link diagrams are said to be \emph{homotopic} if they are related by a finite sequence of 
    \emph{flat Reidemeister moves} shown in Figure~\ref{fig:frmove}
    along with ambient isotopies of the surface.
    \label{def:fr}
\end{definition}
\begin{figure}[ht]
    \centering
    \includegraphics[scale=0.9]{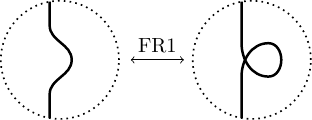}
    \hspace{16pt}
    \includegraphics[scale=0.9]{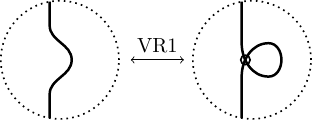}\\
    \vspace{16pt}
    \includegraphics[scale=0.9]{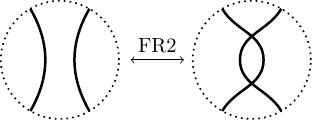}
    \hspace{16pt}
    \includegraphics[scale=0.9]{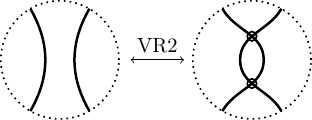}\\
    \vspace{16pt}
    \includegraphics[scale=0.9]{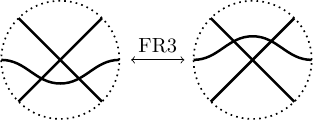}
    \hspace{16pt}
    \includegraphics[scale=0.9]{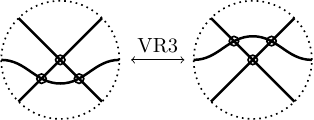} \\
    \vspace{16pt}
    \includegraphics[scale=0.9]{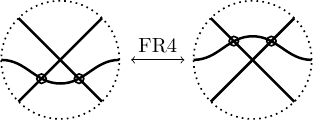}
    \caption[Flat Reidemeister moves]{Flat Reidemeister moves}
    \label{fig:frmove}
\end{figure}

The \emph{crossing number} of a flat link diagram $D$ is denoted $cr(D)$ and defined as the number of flat crossings in $D$. (Note that we ignore virtual crossings here.)
For a flat link $\alpha$, $cr(\alpha)$ is defined as the minimum $cr(D)$, over all diagrams $D$ for $\alpha$. Thus, the crossing number is a flat link invariant.

As shown in Figure~\ref{fig:kishino}, given a virtual knot diagram, there is an associated flat knot diagram given by flattening
all the crossings. This induces a well-defined projection from virtual knots to flat knots.
When we refer to a component of a flat link, we mean a connected component of a virtual link overlying it. In this paper, we will be mainly interested in flat knots, which are flat links with one component.

\begin{definition}
For an oriented flat knot diagram with $n$ classical crossings, its \emph{flat Gauss diagram}
is a counterclockwise oriented skeleton with $2n$ points on the skeleton and $n$ arrows.
Every arrow encodes a crossing as shown in Figure~\ref{fig:fgauss2}.
The order of the points on the skeleton tells us the adjacency of the crossings in the flat knot diagram.
\end{definition}

\begin{definition}
    A \emph{flat Gauss code} (or \emph{flat Gauss word}) is a notation for representing flat Gauss diagrams. 
    From 12 o'clock, travel counterclockwise around the skeleton and assign successive numbers $1,2,\ldots$ to each arrow encountered.
    At each arrowhead (or tail), record ``{\tt U}'' (or ``{\tt O}''), along with the arrow number.
    The Gauss code is a recording of every numbered arrowhead and tail, starting again at 12 o'clock and going counterclockwise around the skeleton.
\end{definition}
\begin{figure}[ht]
    \centering
    \includegraphics[scale=0.9,trim=0 0 0 0]{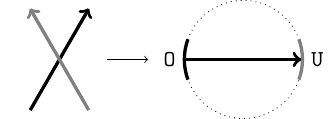}
   \caption{Flat crossings to Gauss diagram arrows}
    \label{fig:fgauss2}
\end{figure}

\begin{figure}[ht]
    \centering
    \includegraphics[angle=0,scale=1]{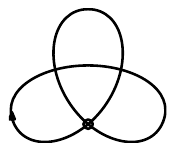}
    \;
    \includegraphics[angle=0,scale=0.9,trim=0 0 0 0]{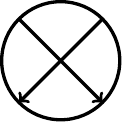}

    \caption{Flat virtual trefoil and its Gauss diagram.}
    The corresponding Gauss code is {\tt O1U2U1O2}.
    \label{fig:gauss5}
\end{figure}
Flat knot diagrams are completely determined by the associated Gauss diagram, 
up to VR1, VR2, VR3 and FR4.
On Gauss diagrams, the corresponding Reidemeister moves are as shown in Figure~\ref{fig:gfrmove}.
Note that VR1, VR2, VR3, FR4 do not change the flat Gauss diagram.

\begin{figure}[ht!]
    \centering
    \includegraphics[scale=0.9]{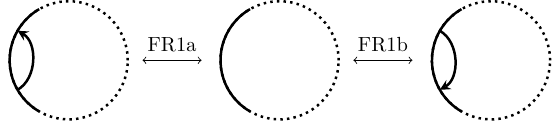}\\
    \vspace{8pt}
    \includegraphics[scale=0.9]{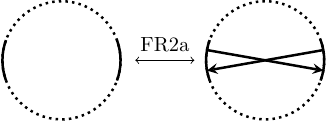} \\
    \hspace{8pt}
    \includegraphics[scale=0.9]{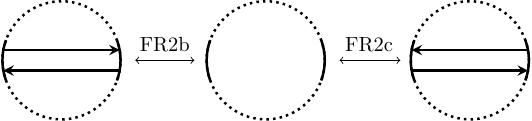} \\
    \vspace{11pt}
    \includegraphics[scale=0.9]{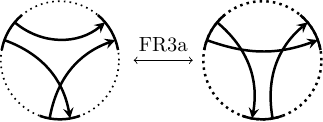}
    \hspace{11pt}
    \includegraphics[scale=0.9]{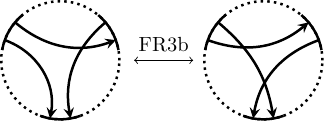}\\
    \vspace{11pt}
    \includegraphics[scale=0.9]{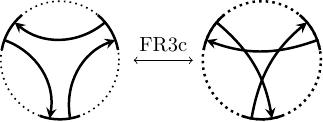}
    \hspace{11pt}
    \includegraphics[scale=0.9]{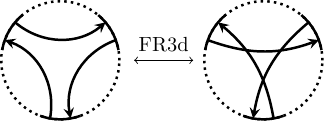}
   \caption{Flat Reidemeister moves on Gauss diagrams}
    \label{fig:gfrmove}
\end{figure}

\begin{definition}
    An \emph{immersion representation} (or a \emph{diagram on surface $\Sigma_g$}) of a flat knot $\alpha$ is an immersion $\omega_{\alpha}:S^1\looparrowright \Sigma_g$, where
    $\Sigma_g$ is a connected, oriented, closed surface of genus $g$.
    Two immersion representations are \emph{stably equivalent} if they are related by a finite sequence of stabilizations, destabilizations, and homotopies.
    The \emph{flat genus} $g(\alpha)$  of a flat knot $\alpha$ is defined to be the minimal genus over all surfaces admitting an immersion representation for $\alpha$.
    \label{def:fkim}
\end{definition}

We can get a planar projection of $\Sigma_g$ so that every double point is transverse and realize all the double points that are not in the immersion as virtual crossings. Thus we get a flat knot diagram. 
From a flat knot diagram, 
if we regard the virtual crossing as in Figure~\ref{fig:thickened2} 
(left),
every flat crossing is realized by thickening the surface
as in Figure~\ref{fig:thickened2} (right), 
and then attaching $2$-disks to that surface along its boundaries
to obtain a Carter surface.
By \cite{carterkamadasaito},  there is a bijection between the equivalence classes of flat knot diagrams and those of immersion representations.
Throughout this paper, we will interchangeably use the three representations.

\begin{figure}[ht]
    \centering
    \hspace{11pt}
    \includegraphics[scale=1.2]{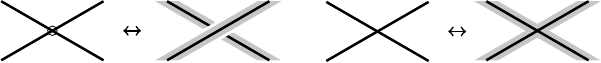}
    \caption{Flat knots as immersed  loops on surfaces}
    \label{fig:thickened2}
\end{figure}

The symmetries of a flat knot are generated by the two involutions.
Let $\alpha$ be a flat knot represented as an immersion $\omega_\alpha:S^1 \looparrowright \Sigma_g.$
\begin{enumerate}
    \item The \emph{reverse} $-\alpha$ is given by changing the orientation of $S^1$;
    \item The \emph{mirror image}  $\alpha^*$ is given by changing the orientation of $\Sigma_g$;
    \item The \emph{reversed mirror image}  $-\alpha^*$. 
\end{enumerate}

On flat Gauss diagrams, the operation $\alpha \to -\alpha$ reverses the orientation of the skeleton, whereas the operation $\alpha \to \alpha^*$
reverses all the arrows.
\begin{figure}[ht]
   \centering
   \includegraphics[scale=1.1]{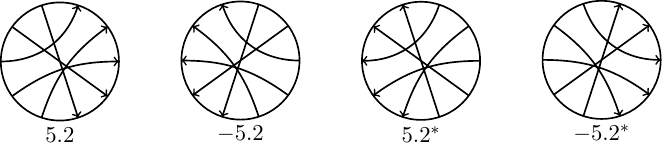}
   \caption{Gauss diagrams of `siblings' of the flat knot 5.2}
   \label{fig:fk52}
\end{figure}
\begin{figure}[ht!]
   \centering
   \includegraphics[scale=0.35]{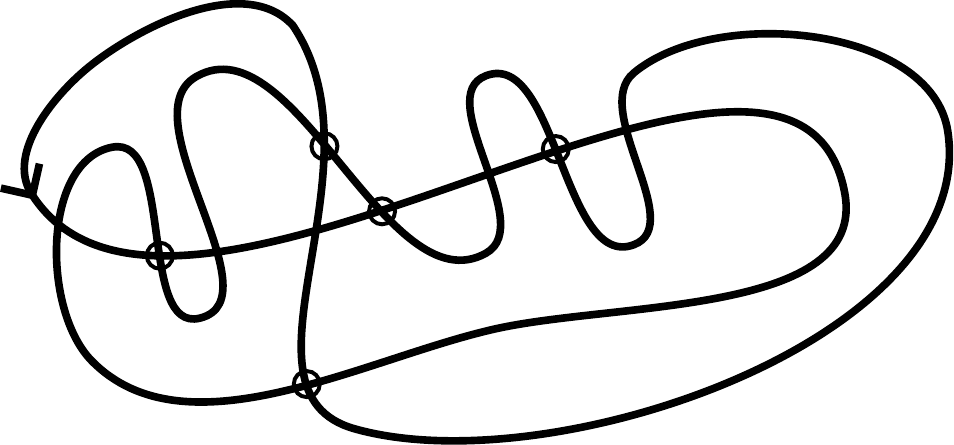}\;\;
   \includegraphics[scale=0.35]{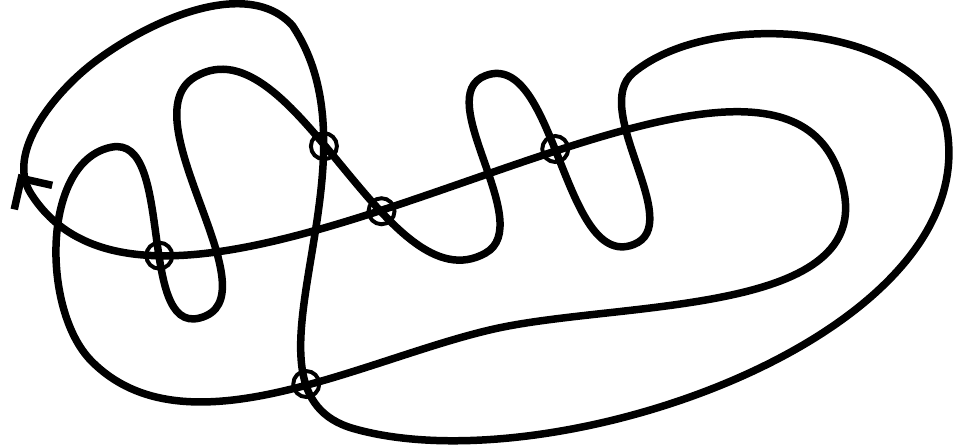}\\
   5.2 \hspace{0.4\textwidth} $-5.2$\\
   \includegraphics[scale=0.35]{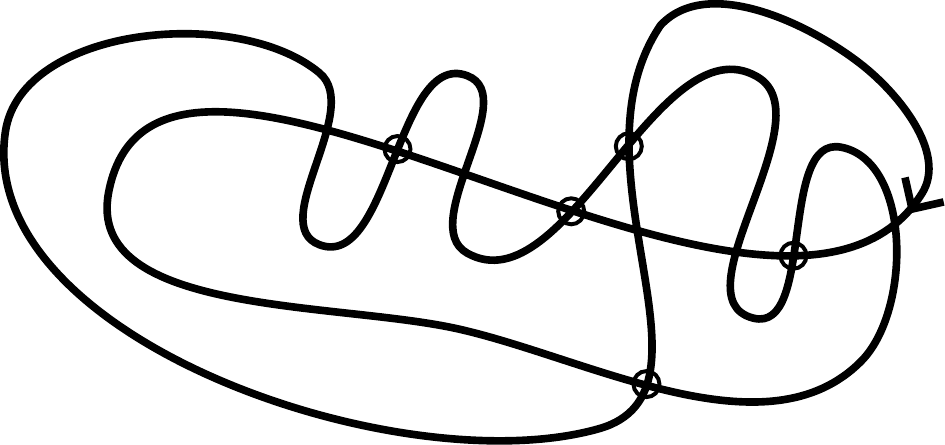}\;\;
   \includegraphics[scale=0.35]{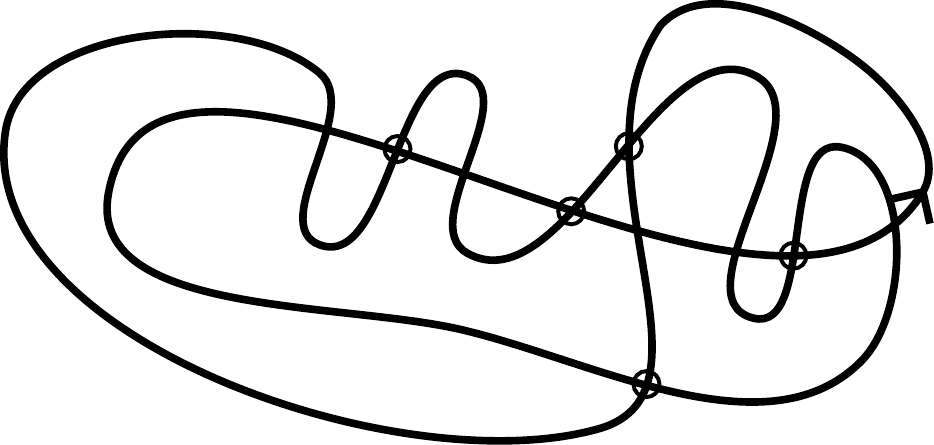}\\
   $5.2^*$ \hspace{0.4\textwidth} $-5.2^*$
   \caption{Diagrams of `siblings' of the flat knot 5.2}
   \label{fig:fk522}
\end{figure}

\begin{definition}
    The five symmetry types of flat knots are
    chiral, reversible,  $+$-achiral, $-$-achiral and fully-achiral as defined in Table~\ref{table:symmetry}.

    \begin{table}[ht!]
    \centering
\begin{tabular}{|c |c |c |c |c |c|} 
 \hline
  & chiral & reversible &  $+$-achiral & $-$-achiral & fully-achiral \\ [0.5ex] 
 \hline
$\alpha=-\alpha$    & No & Yes & No  & No  & Yes \\\hline
$\alpha=\alpha^*$   & No & No  & Yes & No  & Yes \\\hline
$\alpha=-\alpha^*$  & No & No  & No  & Yes & Yes \\
 \hline
\end{tabular}
\vspace{2mm}
\caption{Symmetry types of flat knots}
\label{table:symmetry}
\end{table} 
\end{definition}

On FlatKnotInfo \cite{flatknotinfo}, we consider the `siblings' as of the same knot type,
and we use the letters `c', `r', `$+$', `$-$' and `a' to indicate the symmetry type, respectively.

We now introduce the index of a flat crossing and use it to define when a flat knot or flat knot diagram is mod $p$ Alexander numberable, almost classical, or checkerboard colorable. As a historical note, these notions were first introduced for virtual knots by Silver and Williams, who coined the term \emph{almost classical} in \cite{MR2275098}. It is straighforward to see that these notions are well-defined for flat knots, and further that a virtual knot diagram is mod $p$ Alexander numberable (or almost classical or checkerboard colorable) if and only if its corresponding flat knot diagram is.

\begin{figure}[ht]
    \centering
    \includegraphics[scale=1.3]{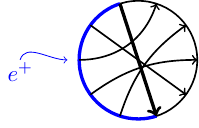}
    \hspace{11pt}
    \includegraphics[scale=0.4]{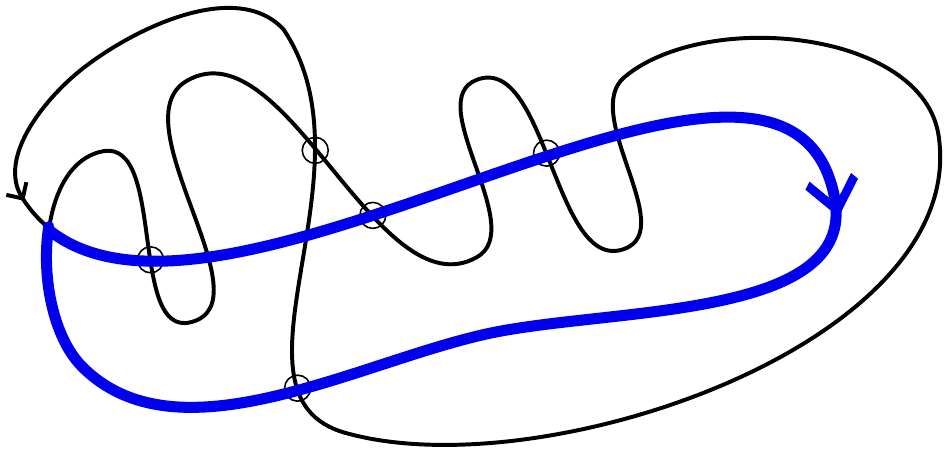}
   \caption{Loops $e^+$ associated to crossing $e$}
    \label{fig:indexdef2}
\end{figure}

Let $D$ be a Gauss diagram and $\arr(D)$ be its set of arrows.
For an arrow $e\in\arr(D)$, let $e^+$ denote open interval on the flat Gauss diagram from the arrow tail to the arrow head
as shown in Figure~\ref{fig:indexdef2}.

\begin{definition}
    In a Gauss diagram,
    the \emph{index} $n(e)$ of an arrow $e$ is given by 
    the number of arrow tails in $e^+$ minus the number of arrow
    heads in $e^+$.
    \label{def:index}
\end{definition}

For example, the indices of arrows $a,b,c,d,e$ in Figure~\ref{fig:indexdef0} are $4,2,0,-3,-3,$, respectively. 

\begin{figure}[ht]
    \centering
    \includegraphics[scale=2]{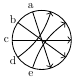}
   \caption{Gauss diagram for 5.1 with arrow labelling}
    \label{fig:indexdef0}
\end{figure}

Given a flat knot diagram, using the one-to-one correspondence between the set of flat crossings and arrows in its Gauss diagram, Definition \ref{def:index} applies equally well to define the index at the flat crossings. 

\begin{definition}
    A flat knot diagram is said to be \emph{mod $p$ Alexander numberable} if all its crossings have index $n(e) \equiv 0 \mod p$.  
    When $p=0$, the diagram is said to be \emph{almost classical} and when $p=2$, it is said to be \emph{checkerboard colorable}. A flat knot is said to be \emph{mod $p$ Alexander numberable} (or \emph{almost classical} or \emph{checkerboard colorable}) if it admits a flat knot diagram  having the property.
    \label{def:acflat}
\end{definition}

When a flat knot is represented as an immersed curve in a surface, the arrow indices can be viewed as intersection numbers. Given an immersion representation $\omega_\alpha : S^1\looparrowright \Sigma_g$ of $\alpha$, the image of $e^+$ is an oriented loop in $\Sigma_g$. 
Referring to Figure~\ref{fig:indexdef2}, let $[e^+]$ denote the homology class of the image of $e^+$ in $H_1(\Sigma_g,\Z).$ 
The index $n(e)$ can be viewed as the intersection number of the $1$-cycles  in $\Sigma_g$ associated to the crossing and the core element $s=[\omega_\alpha(S^1)]$.
Referring to \cite{turaev04}*{Section~4.2}, we note that the set $$\{ s\} \cup \{ [e^+]  \mid e \in \arr(\alpha) \}$$ gives a set of generators for $H_1(\Sigma_g;\Z)$.
Therefore, if $n(e)=0$ for all $e\in \arr(\alpha)$, then $\omega_\alpha(S^1)$ is homologically trivial in $\Sigma_g$. 
Likewise, if $n(e) \equiv 0 \mod 2$ for all $e\in \arr(\alpha)$, then $\omega_\alpha(S^1)$ is $\Z/2$ homologically trivial. Thus,  a flat knot is almost classical if and only if it admits a homologically trivial representative $\omega_\alpha:S^1 \to \Sigma_g$, and it is checkerboard colorable if and only if it admits a $\Z/2$ homologically trivial representative.  

Given a flat knot $\alpha$ that is mod $p$ Alexander numberable, not every diagram for $\alpha$ will necessarily be mod $p$ Alexander numberable. However, just as for virtual knots,  any \emph{minimal crossing} diagram for $\alpha$ will be mod $p$ Alexander numberable. A proof of this fact for virtual knots can be found in \cite{boden_acknot}, and the proof carries over to flat knots without issue. This feature is especially useful in light of the fact that the tabulation of flat knots in \cite{flatknotinfo} is given in terms of minimal crossing representatives. Consequently, one can completely determine whether any flat knot is mod $p$ Alexander numberable, almost classical, and/or checkerboard colorable by simply computing the arrow indices of the representative diagram.

\section{Monotonicity and classification} \label{sec-mono}
In this section we review the theoretical basis underlying the classification of flat knots and we introduce an algorithm for implementing it.

In \cite{turaev04}, Turaev developed an algorithm for classifying flat knots.
It has been implemented by Gibson \cite{gibson} and is the basis for FlatKnotInfo \cite{flatknotinfo}.
This algorithm leverages the monotonicity property for flat knots, as stated in Theorem~\ref{thm:monotonicity} below.
(Note that the corresponding statement is not true for classical knots. In general, one may need to increase the number of crossings in any sequence of Reidemeister moves relating two minimal crossing diagrams of the same classical knot. It is an open problem to find a good upper bound on this number.)
\begin{theorem}[Monotonicity \cites{hassscott,manturovbook, cahn,david}]
    For any flat knot diagram, there exists a sequence of flat Reidemeister moves such that the number of crossings monotonically decreases until one achieves a minimal crossing diagram.
    Further, any two minimal crossing diagrams of the same flat knot are related by a sequence of FR3 moves.
    \label{thm:monotonicity}
\end{theorem}

Monotonicity provides the following general scheme for classifying flat knots. Given a flat knot diagram, we first apply Reidemeister moves to reduce the crossing number. After a finite number of reductions, one obtains a minimal crossing diagram. Next, one determines all diagrams related to the minimal crossing diagram by Reidemeister $3$-moves. Since any two minimal crossing diagrams of the same flat knot are related by Reidemeister $3$-moves, this set of minimal crossing diagrams can be used to completely classify the flat knot type.

This scheme can be implemented as an algorithm for tabulating flat knots up to $n$-crossings. The first step is to construct all flat knot diagrams up to $n$-crossings. Each flat knot diagram is reduced to a minimal crossing diagram, which is possible by monotonicity. The next step is to determine the Reidemeister $3$ orbit of each minimal crossing diagram and record a unique representative for each one. This step uses a linear ordering on the set of flat knot diagrams and results in a unique ``name'' for each flat knot. The last step is to validate the results by calculating enough invariants of the flat knots to distinguish each pair of flat knots in the table.

Gibson applied this method to tabulate flat knots up to 4 crossings in \cite{gibson}.  He represented flat knots as nanowords (cf.~\cite{turaev06}) and used the $u$-polynomial, the $\phi$-invariant, and the 2-parity projection to distinguish the flat knot types.  This approach works for flat knots with up to 4 crossings, but the invariants are not sufficiently powerful to distinguish flat knots with five or more crossings.

In representing flat knots, we will use Lyndon words, defined as follows.

\begin{definition}
The \emph{{\tt OU} word} of a Gauss code is the binary string in $\{ {\tt O, U} \}$ obtained by removing the integers. For example, the flat knot with Gauss code {\tt O1O2O3U1U3U2}
has {\tt OU} word {\tt OOOUUU}.
We call an
{\tt OU} word \emph{Lyndon} if it is minimal up to cyclic rotation. Here the ordering has ${\tt O} < {\tt U}$. 
\label{def:lyndon2}
\end{definition}

The enumeration algorithm for flat knots \cite{flatknotinfo} starts with Lyndon words and then considers all possible matchings on them. One advantage to this approach is there is a linear time algorithm for generating Lyndon words due to Duval \cite{duval-1983}.
The Gauss diagram also determines a \textit{matching}, which is a permutation indicating how each ``{\tt O}'' is matched to the corresponding ``{\tt U}''. In fact, the Gauss diagram is completely determined by its {\tt OU} word and matching. 

To derive the matching, our convention is to first number the {\tt O}'s sequentially going counterclockwise from 12 o'clock in the diagram, and then to record the numbers of the corresponding {\tt U}'s, again going counterclockwise from 12 o'clock. Note that the ordering of crossings in the matching can differ from the ordering in the Gauss word. For example, in the matching of the flat knot diagram in Figure~\ref{fig:lyndon2},  the order of the last two crossings is switched.
In FlatKnotInfo \cite{flatknotinfo}, we name the flat knot by its Lyndon word ordering. For example, the flat knot 5.1 has the least Lyndon word representation among all $5$ crossing flat knots.

\begin{figure}[ht!]
    \centering
    \includegraphics[angle=0,scale=1.4,trim=0 0 0 0]{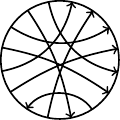}\\
    \caption{\small The flat knot  8.65741 has Gauss word {\tt O1O2O3O4O5O6U7O8U2O7U4U1U6U3U8U5}.
   Thus its {\tt OU}-word is {\tt OOOOOOUOUOUUUUUU}
   and its matching is {\tt [8 2 4 1 6 3 7 5]}.}
    \label{fig:lyndon2}
\end{figure}

The Lyndon word method is especially useful in generating the tables of checkerboard colorable flat knots and almost classical flat knots.
For those flat knots, the tables include higher crossing flat knots by using the ``{\tt OU}''-pattern:
Along the flat knot Gauss diagram we label the singular points $1,2,\dots,2n$.
If two singular points are matched as the head and tail of an arrow, then they should have even and odd numbers and be non-consecutive. We assume the tails are odd points and pair them up with the even points. We can always assume the first arrow has the least head-tail difference. 
This allows us to consider much fewer than $(n-3)( (n-1)! ) $ cases.
On each pattern, we can flip the head and tail of the arrow to obtain a checkerboard colorable diagram, so the case number is multiplied by $2^n$.
For each one, we apply Reidemeister moves until we obtain a minimal diagram and then we find its 3-orbit and minimal representative (in the linear ordering).

An alternative approach for classifying flat knots was developed by Chu \cite{chu}. Those methods apply to \textit{long} flat knots, and the classification is achieved in terms of canonical diagrams, which need not have minimal crossing number. It is not clear how to adapt those methods to the case of \textit{round} flat knots considered here.

\section{Based matrices}
In this section, we review the basic notions of based matrices from \cite{turaev04} and recall the definition of the $\phi$-invariant from \cite{gibson}.  The $\phi$-invariant is a universal invariant of based matrices under equivalence, and it gives a strong invariant of flat knots. We introduce two other invariants of flat knots, the inner and outer characteristic polynomials, and they are also derived from the based matrix. The characteristic polynomials are interesting as they are flat knot analogues of the classical Alexander polynomial.

\subsection{Primitive based matrices}
Let $\alpha$ be a flat knot, realized both as a Gauss diagram $D$ and as an immersion $\omega_\alpha: S^1\rightarrow \Sigma_g$, where    $\Sigma_g $ is an oriented closed surface of genus $g$. Let $b:H_1(\Sigma_g,\Z)\times H_1(\Sigma_g,\Z)\rightarrow \Z$ be the skew-symmetric form given by the intersection pairing.
Let $\arr(D)$ be the set of arrows in $D$ and set
$G=\left\{ [e^+]\in H_1(\Sigma_g; \Z) \mid e\in \arr(D) \right\}$,
where $e^+$ is the oriented loop in $\Sigma_g$. Figure~\ref{fig:indexdef2} shows two representatives of $e^+$ in blue: on the left it appears  on the Gauss diagram as the arc from the tail to the head of arrow $e$; on the right it appears as the oriented loop in the flat knot diagram.
The core element is defined to be $s=[\omega_\alpha(S^1)]\in H_1(\Sigma_g; \Z)$ and $\bar G = \{s\} \sqcup  G$. Both $G$ and $\bar G$ are regarded as ordered sets once an ordering of the crossings has been fixed with the understanding that the core element is always first in $\bar G$.

\begin{definition}
   The \emph{based matrix} $T(D)$  associated to the triple $(G,s,b)$ with $|G|=n$ 
   and a fixed ordering of elements in $G$
    is the $( n+1 )\times (n+1)$ skew-symmetric matrix over
    $\Z$, where the $i,j$-entry of $T(D)$ is the intersection pairing of the $i$-th and
    $j$-th element of $\bar G = \{s\} \sqcup  G$.
    \label{def:basedmatrix}
\end{definition}

We now introduce algebraic analogues of the first and second Reidemeister moves for based matrices.

\begin{definition}
    Let $T=(G,s,b)$ be a based matrix. Then we say: 
\begin{itemize}
    \item $x\in G$ is an \emph{annihilating element} if $b(x,y)=0$ for all $y\in  G\sqcup\left\{ s \right\}$.
    \item $x\in G$ is a \emph{core element} if $b(x,y)=b(s,y)$ for all $y\in G$.
    \item  $x,y\in G$ are \emph{complementary elements} if $b(x,z)+b(y,z)=b(s,z)$ for all $z\in G\sqcup\left\{ s \right\}$.
\end{itemize}
A based matrix $T=(G,s,b)$ is said to be \emph{primitive} if $G$ does not contain annihilating, core, or complementary elements. 
An \emph{elementary reduction} of the based matrix $T=(G,s,b)$ is the operation of removing from $G$
an annihilating element, a core element, or a pair of complementary elements. The inverse operation is called an \emph{elementary extension}.
\label{def:primitive}
Two based matrices are said to be \emph{homologous} if one can be obtained from the other by a finite sequence of elementary extensions/reductions and isomorphisms.
\end{definition}
Referring to \cite{turaev04}*{Section~6.1}, every skew-symmetric square matrix over an abelian group determines a based matrix.
Every based matrix is obtained from a primitive based matrix by elementary extensions. Two primitive based matrices are isomorphic if and only if they are homologous.
For a based matrix $T$,  we use $T_\bullet$ to denote the associated primitive based matrix obtained from $T$ under elementary reduction.

Turaev in \cite{turaev04}*{Lemma~4.2.1} gave an algorithm for calculating the based matrix $T(D)$ with respect to the Gauss diagram $D$. We illustrate it with an example.

\begin{example}
Consider  the flat knot $5.1$, whose Gauss diagram appears in Figure~\ref{fig:indexdef0}. Using the indicated ordering of the crossings, we compute that it has based matrix with respect to  $\bar G=\left\{s,[a^+],[b^+],[c^+],[d^+],[e^+] \right\}$ given by
$$\begin{bmatrix}
 0& -4& -2&  0&  3&  3\\
 4&  0&  1&  2&  4&  3\\
 2& -1&  0&  1&  3&  2\\
 0& -2& -1&  0&  2&  1\\
-3& -4& -3& -2&  0&  0\\
-3& -3& -2& -1&  0&  0
\end{bmatrix}.
$$
Note that FlatKnotInfo \cite{flatknotinfo} provides calculations of the based matrices for all the tabulated flat knots. In addition, a Python program that calculates the based matrix from a Gauss diagram can be found in \cite{chen-thesis}*{Appendix B}.
\end{example}

Next, we introduce the $u$-polynomial, which was first defined by Turaev \cite{turaev04}. In the following, we use $n(e)$ to refer to the arrow index of $e$ (see Definition \ref{def:index}).

\begin{definition}
    Let $\alpha$ be a flat knot with Gauss diagram $D$.
    Then the $u$-polynomial of $\alpha$ is defined as
    $$
    u_\alpha(t)=\sum_{e\in \arr(D)} \text{sign}( n(e)) t^{|n(e)|}.
    $$
    \label{def:upoly}
\end{definition}

\begin{definition}
    Given a flat knot $\alpha$, its \emph{$r$-th covering} is denoted $\alpha^{(r)}$ and defined to be the flat knot obtained from a Gauss diagram $D$ of $\alpha$ after deleting arrows $\left\{ e\in \arr(D)\mid n(e)\notin r\Z \right\}$.
    \label{def:covering}
\end{definition}

The \emph{$r$-th covering} $\alpha^{(r)}$ of a flat knot was introduced in \cites{turaev04,turaev-vk-cob}, where it was shown that the $r$-th covering is a flat knot invariant representing
 the lift of $\omega_\alpha$ to the $r$-fold covering of $\Sigma$ induced by the dual in $H^1(\Sigma;\Z/r)$ of $[\omega_\alpha(S^1)]$.
 
Manturov gave a purely combinatorial description of the $r$-th covering of a flat knot in terms of Gaussian parity projection in  \cites{manturov-parity}.
A detailed explanation of the correspondence between lifts to abelian covers and Gaussian parity projection can be found in \cite{boden_sign}*{Section~5.3}.

\subsection{The $\phi$-invariant}
By choosing a different ordering of $G$, one will obtain a different based matrix. Therefore, the based matrix depends on the choice of a diagram and ordering of $G$. 
Gibson \cite{gibson} defined the $\phi$-invariant to record the information in a primitive based matrix. 

Let $D$ be a flat knot diagram, and $ T_\bullet (D)$ be a primitive based matrix of $\alpha$. 
Define $\varphi(T_\bullet (D))$ to be the entries of the columns below the diagonal of $T_\bullet (D)$.

For example, if
$$ T_\bullet(D) =
\begin{bmatrix}
0 & 1 & 0 & 0 & -1\\
-1 & 0 & 1 & 0&-2\\
0&-1&0&1&0\\
0&0&-1&0&1\\
1&2&0&-1&0
\end{bmatrix}
,
$$
then $\varphi( T_\bullet (D))=[-1,0,0,1,-1,0,2,-1,0,-1]$.

Let $\ds\phi_\alpha= \min\left\{ \varphi( T_\bullet) \mid T_\bullet\text{ is a primitive based matrix of a diagram of } \alpha   \right\}$,
where the minimum is taken with respect to the lexicographic order.
For example, $[1,2,3]<[2,2,3]<[2,3,3]$.
Refer to \cite{chen-thesis}*{Section~3.2} for the algorithm to calculate  $\phi_\alpha$.

The based matrix is a very powerful invariant. 
It can be used to separate symmetric siblings of a flat knot in many cases.
For example, using the flat Gauss diagrams of  
the flat knot $5.2$, we obtain four different $\phi$-invariants.
in Figure~\ref{fig:fk52}.
They are
\begin{align*}
\phi_\alpha&=[-3, -2, -1, 2, 4, 1, 1, 2, 3, 1, 3, 4, 1, 2, 1],\\
\phi_{-\alpha}&=[-4, -2, 1, 2, 3, 1, 3, 2, 4, 2, 1, 3, 0, 1, 0],\\
\phi_{\alpha^*}&=[-3, -2, -1, 2, 4, 0, 1, 3, 4, 0, 1, 2, 2, 3, 1],\\
\phi_{-\alpha^*}&=[-4, -2, 1, 2, 3, 1, 2, 4, 3, 1, 3, 2, 1, 1, 1].
\end{align*}
Therefore, the flat knot $5.2$ and its siblings are all distinct.

Using the $\phi$-invariant alone, we can distinguish flat knots up to 4 crossings as shown in  Table~\ref{table:phi_dis}.
\begin{table}[ht]
    \centering
\begin{tabular}{||c |c |c| c||} 
 \hline
 Crossings &  \# Flat knots  & \# Non-distinguished by $\phi$ \\ [0.5ex] 
 \hline\hline
3 & 1 & 0 \\
 \hline
4 & 11 & 0\\
 \hline
5 & 120 & 8 \\
 \hline
6 & 2086 & 74 \\
 \hline
7 & 46233 & 1375 \\
 \hline
\end{tabular}
\vspace{2mm}
\caption{Distinguishing  flat knots using $\phi$}
\label{table:phi_dis}
\end{table} 

The first non-distinguished flat knots are the two pairs
$\{5.47, 5.65\}$ and $\{5.89, 5.104\}$ of 5-crossing flat knots in Figure~\ref{fig:fk547}.
Their minimal $\phi$-invariants, up to symmetry, are given by 
$$\phi_{5.47}=\phi_{5.65}=[-2, -1, 0, 1, 2, -1, 1, 1, 3, 1, 0, 1, 0, 1, 0],  $$
$$\phi_{5.89}=\phi_{5.104}=[-1, -1, 0, 1, 1, -1, 0, 1, 1, 0, 1, 1, 1, 1, -1].  $$

\begin{figure}[ht]
   \centering
   \includegraphics[scale=1.1]{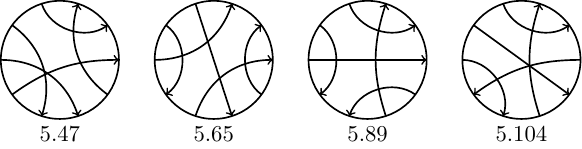}
   \caption{Two pairs of non-distinguished flat knots with 5 crossings}
   \label{fig:fk547}
\end{figure}

Up to 5 crossings, every flat knot has a nontrivial primitive based matrix. The first examples of flat knots with trivial primitive based matrix occur among the $6$-crossing flat knots, namely 6.129 and 6.899 in Figure~\ref{fig:fk6129fk6899}. These flat knots cannot be distinguished from the unknot by the $\phi$-invariant.

\begin{figure}[ht]
   \centering
   \includegraphics[scale=1.1]{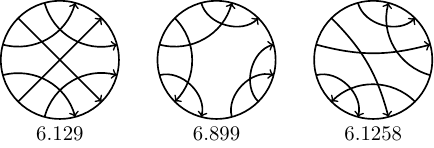}
   \caption{Three 6-crossing flat knots with trivial primitive based matrix}
   \label{fig:fk6129fk6899}
\end{figure}

Note that in Table~\ref{table:phi_dis}, the invariants of the flat knots with $n$ crossings are compared to those of  
all flat knots with $n$ or fewer crossings.
Specifically, there are three $6$-crossing flat knots with trivial $\phi$-invariant, and $19$ $7$-crossing flat knots with trivial $\phi$-invariant. In particular, these flat knots cannot be distinguished from the unknot using only $\phi$-invariants.

\subsection{Characteristic polynomials}
In this subsection, we introduce the inner and outer characteristic polynomials of a flat knot.
The motivation stems in part from the complexity of the $\phi$-invariant and the desire for a more easily computed invariant. It is also of interest since they can be viewed as flat knot analogues of the classical Alexander polynomial.

\begin{definition}
    Let $T$ be a primitive based matrix and set $ P_T(t) = \det(T-tI)$, the characteristic polynomial of  $T$. 
     Further, let $\widehat{T}$ be the matrix obtained from $T$ by deleting its first row and column, and set
     $p_T(t) = \det(\widehat{T}-tI)$, the characteristic polynomial of $\widehat{T}$.
     Then $p_T(t)$ and $P_T(t)$ are called the inner and outer characteristic polynomials of $T$, respectively. 
    \label{def:charpoly}
\end{definition}

The next result shows that inner and outer characteristic polynomials are invariant under isomorphism of primitive based matrices.

\begin{proposition} \label{prop-iso-inv}
 Let $T$ and $T'$ be primitive based matrices. If $T$ and $T'$ are isomorphic,
 then $P_T(t)=P_{ T' }(t)$ and $p_T(t)=p_{ T' }(t)$.
\end{proposition}

\begin{proof}
    Recall that isomorphism of primitive based matrices is defined as a congruence by permutation matrices. But two matrices that are congruent by permutation matrices are necessarily conjugate.
    Since the characteristic polynomial is  invariant under conjugation, it follows that
    $P_T(t)=P_{ T' }(t)$.
    Note that the permutation sends the first element of the ordered set $\bar G$ to the first element of $\bar G'$.
    Therefore, a similar argument shows that $\widehat{T}$ and $\widehat{T'}$ are conjugate, and it follows that $p_T(t)=p_{ T' }(t)$.
\end{proof}

In \cite{turaev04},  Turaev proved that any two primitive based matrices for the same flat knot are isomorphic. Therefore, given a flat knot $\alpha$ 
with primitive based matrix $T$, Proposition \ref{prop-iso-inv} implies that the inner and outer characteristic polynomials of $T$ give well-defined invariants of the flat knot by setting 
$ p_\alpha(t) =p_T(t)  $ and $ P_\alpha(t) =P_T(t)$. 

For example, if $U$ denotes the flat unknot, then its primitive based matrix is $[0]$ and its inner and outer polynomials are $p_U(t)=1 $ and $ P_U(t)=t$.

The based matrix of a flat knot can be viewed as an analogue, for flat knots, of the Seifert matrix of a classical knot. Thus, the inner and outer characteristic polynomials are analogues of the classical Alexander polynomial. Indeed, for fibered knots, the Alexander polynomial has a natural interpretation as the characteristic polynomial of the induced map, on homology, of the monodromy of the fibering. Thus, we wonder whether the inner and outer polynomials have a similar topological interpretation.

\section{Sliceness for flat knots}
In this section, we review the notions of sliceness, ribbonness, and algebraic sliceness for flat knots. Based on empirical data, it is tempting to conjecture that every algebraically slice flat knot is slice, and that every almost classical flat knot is algebraically slice. We present counterexamples to both conjectures, namely flat knots that are algebraically slice but not slice, as well as almost classical flat knots that are not algebraically slice. In the tabulation, the first example of an algebraically slice but not slice flat knot has 6 crossings, and the first example of an almost classical flat knot that is not slice has 11 crossings.

\subsection{Definitions}
In \cite{turaev04}, Turaev introduced the notions of sliceness and algebraic sliceness for flat knots. Interestingly, the first example of a non-slice flat knot was discovered earlier by Carter in \cite{carter}.

In this subsection, we recall the basic definitions of slice, ribbon, and  algebraically slice for flat knots.

\smallskip \noindent
{\bf Sliceness.}
We first recall the definition of sliceness for flat knots and phrase it in terms of slice movies. 

\begin{definition}
Let $\alpha$ be a flat knot represented by an immersion $\omega_\alpha : S^1 \looparrowright \Sigma$.
Then $\alpha$ is said to be \emph{slice} if there exists a compact oriented 3-manifold $M$ with $\partial M = \Sigma$
and a properly immersed disk $D \looparrowright M$ whose boundary is $\omega_\alpha(S^1)$.
    \label{def:sliceflat}
\end{definition}

\begin{figure}[ht]
    \centering
    \includegraphics[scale=0.9]{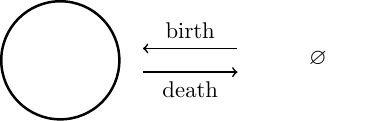}\\
    \includegraphics[scale=0.9]{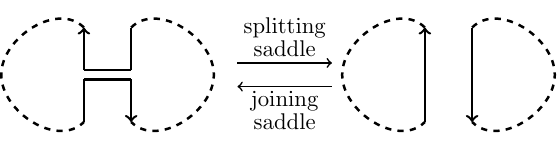}
    \caption{Birth, death, and saddle moves}
    \label{fig:saddle}
\end{figure}

\begin{figure}[ht]
    \centering
    \includegraphics[scale=1.20]{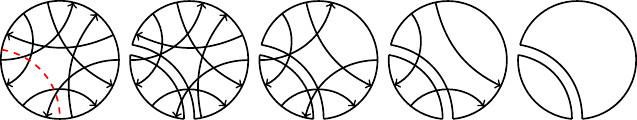}
    \caption{A slice movie for the flat knot 7.45422}
    \label{fig:fk745422}
\end{figure}

Equivalently, a flat knot $\alpha$ is slice if a diagram of $\alpha$ can be transformed to the trivial diagram by a sequence of flat Reidemeister moves, births, deaths and saddle moves shown in Figure~\ref{fig:saddle}. We further require that the number of saddles is equal to the sum of births and deaths, and that the cobordism surface is connected.
Such a sequence is called a \emph{slice movie}.

Clearly, sliceness of a flat knot can be demonstrated by drawing a slice movie, and actually this step can be achieved directly on a Gauss diagram for the knot. For example, Figure~\ref{fig:fk745422} shows how to transform the flat knot 7.45422 into the trivial flat knot using one saddle move and Reidemeister moves on its Gauss diagram. It follows that 7.45422 is slice. This idea of slicing virtual knots directly on their Gauss diagrams is originally due to Robin Gaudreau, and the method works equally well for flat knots. The author is grateful to them for sharing their idea.

\smallskip \noindent
{\bf Ribbonness.}
Next, we introduce a notion of ribbonness for flat knots. As we shall see, this definition of ribbon is different from the one in \cite{turaev04}.

\begin{definition}
A flat knot is said to be \emph{ribbon} if it admits a slice movie with only saddles and deaths.
\end{definition}

For example, since
the slice movie in Figure~\ref{fig:fk745422} has no births, it shows that the flat knot 7.45422 is actually ribbon. For a second example, a standard argument shows that, for any flat knot $\alpha$, the composite knot $-\alpha^*\# \alpha$ is ribbon, see Figure \ref{fig:fk61913}.

It is clear that, if a given flat knot is ribbon, then it is necessarily slice. It is natural to ask whether the converse is true.
\begin{problem}
        Is every  flat knot that is slice also ribbon? 
    \label{prob:sliceribbon}
\end{problem}

This is a flat knot analogue of the famous slice-ribbon conjecture.  

A different notion of ribbon is defined in \cite{turaev04}, and we refer to that here as \emph{strongly ribbon}. A flat knot $\alpha$ is said to be \emph{strongly ribbon} if it admits a Gauss diagram $D$ with $D=-D^*$. In \cite{turaev04}, Turaev proved that if $\alpha$ is strongly ribbon, then it is slice.  In \cite{MR2207902}, Silver and Williams gave an example of a long flat knot that is slice but not strongly ribbon. It is not difficult to see that the example in \cite{MR2207902} is ribbon. In fact, every strongly ribbon flat knot is ribbon, and this can be proved using a nested sequence of saddle moves similar to the slice movie in Figure~\ref{fig:fk61913}. 

\begin{figure}[ht]
    \centering
\includegraphics[scale=0.85,trim=0 0 0 0]{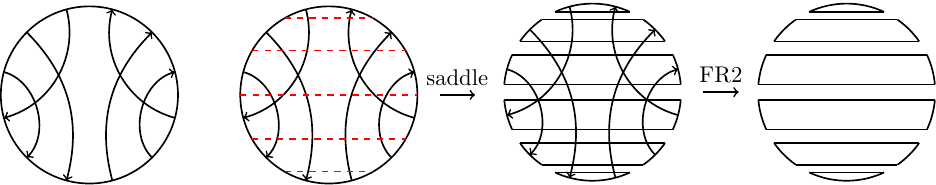}
    \caption{Slice movie for $-\alpha^*\# \alpha$}
    \label{fig:fk61913}
\end{figure}

\smallskip \noindent
{\bf Algebraic sliceness.}
We recall the definitions of the algebraic genus and algebraic sliceness for flat knots, following Turaev \cite{turaev04}. 
\begin{definition}  
    Given a based matrix $T$ with respect to the triple $(G,s,b)$, 
    a \emph{filling} $\chi=\left\{ X_i \right\}_{1\le i\le k}$ is a partition of $G$ such that 
    $G=\bigcup_{1\le i\le k} X_i$ where $X_i\cap X_j=\varnothing$ if $i\neq j$, 
    and $|X_i|\in\left\{ 1,2 \right\}$ for any $i$.
    Let $y_i=\sum_{x\in X_i} x$ (as the formal sum in the free module $\Z G$) and $G'=\left\{  y_i\right\}_{1\le i\le k}$.
    The intersection form $b$ extends to $G'$ by linearity.

Then we obtain a new triple $( G',s, b )$ from this filling, and a new based matrix $\widehat T$ associated with it.
The \emph{genus of a based matrix} $\sigma(T)$ is defined to be   $\frac{1}{2}\min\rank(\widehat T)$ over all possible fillings,
where the $\rank(\,\cdot\,)$ refers to the rank of the integral matrix.
A based matrix $T$ is said to be \emph{null-concordant} if $\sigma(T)=0$.
\label{def:genusmatrix}
\end{definition}

By \cite{turaev04}*{Lemma~7.1.1},  the genus of the based matrix gives an invariant of flat knots called its \emph{algebraic genus}.

\begin{definition}
    The algebraic genus of a flat knot $\alpha$ is denoted $g_a(\alpha)$ and given by the genus $\sigma(T)$ of its based matrix (not necessarily primitive).
    
A flat knot $\alpha$ is said to be \emph{algebraically slice} if $g_a(\alpha)=0$, namely if its
based matrix is null-concordant.  
 \label{def:alg_genus}
\end{definition}

Thus we can visualize the fillings on a Gauss diagram since the generator set of a based matrix corresponds to the arrow set of the Gauss diagram.
As shown in Figure~\ref{fig:k6464_1}, the based matrix of the diagram is 
$$
\begin{bmatrix}
    \gray{0} &\gray{-1}&\gray{1}&\gray{-2}&\gray{0}&\gray{2}\\
\brown{1} & \brown{0}& \brown{1}&\brown{ -2 }& \brown{0}&\brown{2}\\
\brown{-1}&\brown{-1}&\brown{0}&\brown{-2}&\brown{0}&\brown{2}\\
2 & 2& 2& 0& 1&2\\
\blue{0} &\blue{0}&\blue{ 0}&\blue{-1}&\blue{0}&\blue{1}\\
-2&-2&-2&-2&-1&0
\end{bmatrix}
$$
The last five rows correspond to the five arrows from 12 o'clock of the Gauss diagram in counterclockwise order.
The filling $[(1, 2),(3, 5), (4)]$, denoted by dashed brown, black, and dotted blue, respectively, gives the algebraic genus zero.
    When a diagram has a slice movie consisting of only splitting saddles, deaths, FR3 and decreasing FR1 and FR2 moves,
    the FR1 and FR2 moves determine a filling with $\sigma(T)=0$.
For example, after applying a splitting saddle move to the diagram in Figure~\ref{fig:k6464_1}, the arrow pairs (1,2) and (3,5) can be removed by FR2 moves, and the 4th arrow can be removed by an FR1 move.

\begin{figure}[ht]
   \centering
   \includegraphics[scale=1.2]{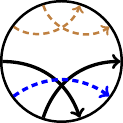}
   \caption{Fillings to show $g_a(5.21)=0$ }
   \label{fig:k6464_1}
\end{figure}

\subsection{Algebraic sliceness does not imply sliceness}
\label{sec:algebraicslice}
In \cite{turaev04}, Turaev proved that any flat knot that is slice is necessarily algebraically slice. It is natural to ask whether the converse is true: is every algebaically slice flat knot also slice? We will find a counterexample to show that the answer is no. In fact, we will find several examples of flat knots that are algebraically slice but not slice.

    Recall the $r$-th covering $\alpha^{(r)}$ of flat knot $\alpha$ in Definition~\ref{def:covering}.
In the proof of \cite{turaev04}*{Corollary~5.17}, Turaev showed that if $\alpha$ is slice, then its $r$-th covering $\alpha^{(r)}$ is also slice.
We use this with $r=3$ to give the first example of a flat knot that is algebraically slice but not slice, see Figure~\ref{fig:k6464}.

\begin{example}
    The flat knot 6.464  is algebraically slice but not slice.
    \label{eg:6464}
\end{example}
\begin{proof}
    Consider the flat knot 6.464 in Figure~\ref{fig:k6464}. Its based matrix can be shown to have genus $0$ since there exists a filling as in Figure~\ref{fig:k6464_2}-left. Thus, the based matrix is null-concordant. Therefore 6.464 has $g_a(\alpha)=0$ and is algebraically slice.
    On the other hand, its $3$-fold covering is the flat knot $-$4.2. By our calculation \cite{flatknotinfo}, the based matrix of 4.2 has genus $1$, and all the corresponding fillings of minimal diagrams are listed in Figure~\ref{fig:k6464_2}-right.
Thus 4.2 is not algebraically slice and not slice. 
Therefore, 6.464 is also not slice.
\end{proof}

An alternative argument to show that the flat knot 4.2 is not slice is to use the fact that its $u$-polynomial is $-t^3+t^2+t$ and to recall that any flat knot that is slice must have trivial $u$-polynomial.
By \cite{flatknotinfo}, up to 6 crossings,
6.464 is the only known flat knot that is algebraically slice but not slice.
The only other potential example is 6.540, which is algebraically slice but not known to be slice. In fact, 6.540 is the only flat knot up to 6 crossings whose slice state is unknown.

\begin{figure}[ht]
   \centering
   \includegraphics[scale=0.9]{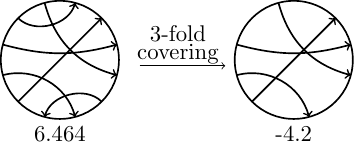}
   \caption{The $3$-fold covering of the flat knot 6.464}
   \label{fig:k6464}
\end{figure}

\begin{figure}[ht]
   \centering
   \includegraphics[scale=0.85]{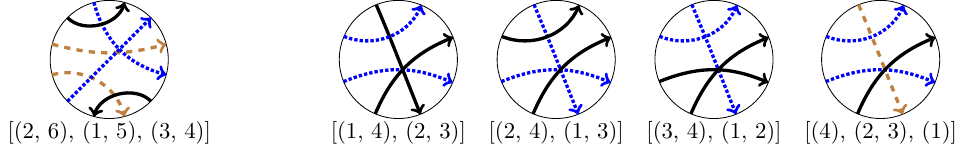}
   \caption[Fillings showing that $g_a(6.464)=0$ and $g_a(4.2)=1$]{Fillings showing that $g_a(6.464)=0$ (left) and $g_a(4.2)=1$ (right)}
   \label{fig:k6464_2}
\end{figure}

    A slice obstruction arising solely from based matrices is regarded as a  \emph{primary obstruction}.
In \cite{turaev04}*{Question~2}, Turaev asks if one can detect non-slice flat knots using \emph{secondary obstructions}, see \cite{turaev04}*{Section~8.4}.
Note that the flat knot 6.464 gives an example, and it is the first one in the tabulation of \cite{flatknotinfo}.
By \cite{turaev04}*{Lemma~8.4.1}, the arrows annihilated by the core element should form a slice flat knot. Therefore, the flat knot 6.464 is seen to be non-slice by the secondary obstructions.

There are additional examples of flat knots which are algebraically slice but not slice in Figure~\ref{fig:fk7nonslice}. Each is seen to be non-slice using parity projection, and one of them, the knot 7.25725, even has trivial primitive based matrix.

\begin{figure}[ht!]
   \centering
   \includegraphics[scale=0.9]{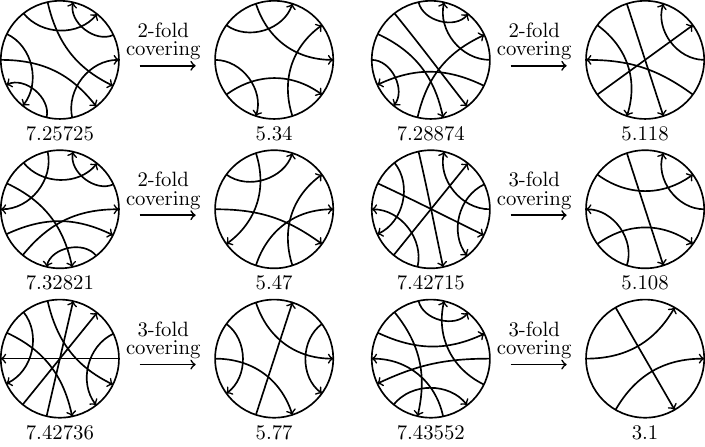}
   \caption{Six algebraically slice flat knots that are not slice}
   \label{fig:fk7nonslice}
\end{figure}

\subsection{Almost classicality does not imply sliceness}
FlatKnotInfo \cite{flatknotinfo} lists slice information for most flat knots up to 7 crossings. These results were obtained by a combination of obstructive and constructive methods. Specifically, we first computed the based matrices and determined which flat knots are algebraically slice. Then we searched among the residual set of flat knots to see if we could find slice movies for them. This approach determined sliceness for all but one flat knot up to 6 crossings, namely the flat knot 6.540. It also worked for the 7-crossing flat knots with the exception of the eight examples in Figure \ref{fig:fk_7slicesuspect}.

In the course of performing these computations, we noticed a pattern, which is that the almost classical flat knots up to 7 crossings all had algebraic genus $g_a=0$ and were all slice.
Based on this, we formulated a conjecture that almost classical flat knots are all slice.

Recall that a flat knot $\alpha$ is said to be \emph{almost classical} if it can be represented by a null homologous curve $\omega_\alpha:S^1 \to \Sigma$. Thus, $\alpha$ is almost classical if and only if it bounds an oriented surface immersed in $\Sigma$. One approach to establishing the conjecture is to perform surgery on the immersed surface to transform it into an immersed disk. (Roughly speaking, this is the approach Levine pioneered in \cite{levine69} for higher dimensional knots.) 

After considerable effort and repeated failures, we decided to check the conjecture on a larger set of examples. For this purpose, we developed a knot slicer program which applies Reidemeister moves and splitting saddle moves in search of a slice movie. There are 1906 almost classical knots with up to 10 crossings, and the program found slice movies for each one. 

With this confirmation, we were even more convinced the conjecture was true. However, attempts to prove it still fell short. In the meantime, we extended the tabulation of almost classical flat knots to 11 and 12 crossings. For instance, there are 
18002 almost classical flat knots with 11 crossings.
When we ran the knot-slicer program over the new sets of flat knots, we found a small subset for which it failed to find a slice movie. We then discovered, to our surprise, that some of these almost classical knots have nonzero algebraic genus. In particular, these flat knots are not slice and give counterexamples to the conjecture.

\begin{example}
Consider the flat knot ac11.7183 in Figure~\ref{fig:ac11}. It has based matrix 
{\footnotesize
$$\begin{bmatrix*}[r]
0&  0&  0&  0&  0&  0&  0&  0&  0&  0&  0&  0\\
0&  0& -1&  1&  0&  1&  0&  0&  0& -1& -1&  1\\
0&  1&  0&  0&  2&  0& -1&  0& -1& -2& -1&  2\\
0& -1&  0&  0&  0&  0&  1&  1&  0&  0&  0& -1\\
0&  0& -2&  0&  0&  2&  0&  1& -1& -2& -1&  3\\
0& -1&  0&  0& -2&  0&  0& -1&  2&  3&  1& -2\\
0&  0&  1& -1&  0&  0&  0&  0&  0&  1&  0& -1\\
0&  0&  0& -1& -1&  1&  0&  0&  0&  1&  0&  0\\
0&  0&  1&  0&  1& -2&  0&  0&  0&  0&  1& -1\\
0&  1&  2&  0&  2& -3& -1& -1&  0&  0&  1& -1\\
0&  1&  1&  0&  1& -1&  0&  0& -1& -1&  0&  0\\
0& -1& -2&  1& -3&  2&  1&  0&  1&  1&  0&  0
\end{bmatrix*}.
$$} 
We claim that this flat knot has algebraic genus
$g_a(\text{ac}11.7183)=1$.
To see this, recall from Definitions~\ref{def:genusmatrix} and \ref{def:alg_genus} that the algebraic genus is determined by considering all possible fillings and taking the one whose associated based matrix has minimal rank. For the above matrix, the filling
$[(1, 11), (2, 10), (3,9),(4, 5), (6,8), (7) ]$ has associated based matrix $T$ with rank 2. Further, none of the other fillings produce based matrices with smaller rank. It follows that $g_a(\text{ac}11.7183)=1$, and this shows that $\text{ac}11.7183$ is not algebraically slice and so not slice.

In fact, among all 11-crossing almost classical flat knots, there are 25 that are not algebraically slice and hence not slice; see Figure~\ref{fig:ac11}. 
\end{example}

\begin{figure}[ht]
    \centering
\includegraphics[angle=0,scale=1.0,trim=0 0 0 0]{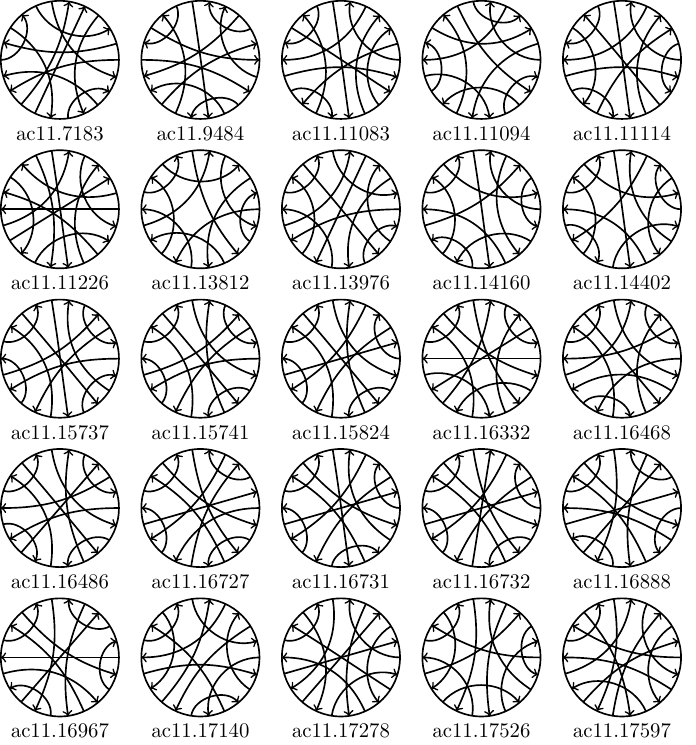}
\caption{Non-slice almost classical flat knots}
\label{fig:ac11}
\end{figure} 
We wonder if there exists a non-slice almost classical flat knot whose Seifert genus is one. Such an example could shed light on the problem of commutativity of long flat knot concordance group.

\section{The flat arrow polynomial}
\label{chap:arr}
In this section, we define the flat arrow polynomial as an invariant of oriented flat links. We then discuss free equivalence, checkerboard colorability and almost classicality. We show that the leading coefficient of the flat arrow polynomial of any checkerboard colorable flat knot is necessarily even. We also introduce the cabling operation for flat knots and use it to give strengthen the arrow polynomial.

\subsection{Definition and basic properties}
In this subsection, we define the flat arrow polynomial of flat knots and links. 
It is closely related to the arrow polynomial of virtual knots and links, which was originally introduced by Dye and Kauffman \cite{virtualarrow}, and independently Miyazawa \cite{miyazawa}, as a powerful generalization of the virtual Jones polynomial \cites{jones, kauffman99}. For convenience, we follow the notational conventions for cusps with multiplicities as used by Miller \cite{miller22}, suitably adapted to the context of flat links. 

The flat arrow polynomial is defined by applying the skein relation of Figure~\ref{fig:skein} to each flat crossing of a flat link diagram $D$. It is an oriented skein relation, so one needs to fix an orientation for $D$ and use it to orient all the edges of $D$. Notice that the local orientations on the edges are preserved under the skein relation of Figure~\ref{fig:skein}. 
At each flat crossing in $D$, we replace each flat crossing by an oriented smoothing and a disoriented smoothing. Virtual crossings are not affected. If $\alpha$ has $n$ flat crossings, then there will be $2^{n}$ states. Each state will consists of loops with only virtual crossings, representing a trivial virtual flat link.

When performing a disoriented smoothing, we introduce hollow triangles, that represent \emph{cusps}  and they freely pass through virtual crossings. Each cusp is a degree two vertex with a choice of vertex orientation. Using the rules in Figure~\ref{fig:arrow2}, the cusps can be reduced and/or combined into multiples until each state consists of a finite number of loops, each with only one cusp with multiplicity $m \in \Z$. Moreover, the total cusp multiplicity for each loop will be even, possibly zero, as we now explain.

Note that each disoriented smoothing gives rise to two cusps, so in any state there are an even number of cusps in total. In fact, the total number of cusps is equal to twice the number of disoriented smoothings. We claim that each loop in the state will have an even number of cusps. To see this, consider the local orientations as one travels around a loop. It switches only at a cusp, and going all the way around and coming back to a start point, one must encounter an even number of cusps since the local orientations must have switched an even number of times. It follows that after reduction and combination, the total cusp multiplicity on each component is even.

\begin{figure}[ht]
   \centering
\includegraphics[scale=0.9]{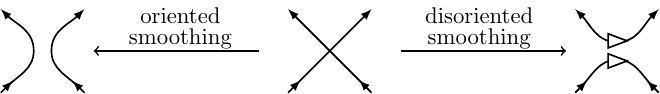}
\caption{Skein relation for the flat arrow polynomial}
   \label{fig:skein}
\end{figure}

\begin{figure}[ht]
   \centering
\includegraphics[scale=0.9]{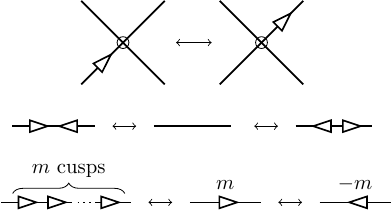}
\caption{Rules for cusp propagation, reduction and combination}
   \label{fig:arrow2}
\end{figure}

If the state $S = C_1 \sqcup \cdots \sqcup C_\ell$ consists of $\ell$ components with $C_i$ having cusp multiplicity $m_i,$ then it evaluates to $\langle \, S \, \rangle = \prod_{i=1}^\ell K_{|m_i|/2}$. The convention here is that $K_0=1,$ namely components with no cusps evaluate to $1$.

\begin{definition}
    For a flat knot or link  diagram $D$, 
    the \emph{flat arrow polynomial} is defined to be $$\ds A(D)=\sum_S (-2)^{|S|-1} \langle \, S \, \rangle,$$
    where the sum is over all states $S$ and $|S|$ denotes the number of loops in $S$.

    The \emph{(normalized) arrow polynomial of flat link $\alpha$} is defined to be  $$\ds \bar A(\alpha)=(-1)^{ cr(D)} A(D) ,$$ where $D$ is a flat knot diagram of $\alpha$.
    \label{def:arrowpoly}
\end{definition}
The next result was originally proved by Kauffman \cite{kauffman12}*{Theorem~8.2}. 

\begin{theorem}[Kauffman]
The normalized flat arrow polynomial is an invariant of oriented flat links taking values in $\Z[K_1,\ldots,K_n]$.
\label{thm:arrowpolyinv}
\end{theorem}

A proof of Theorem~\ref{thm:arrowpolyinv} from first principles can be found in \cite{chen-thesis}. Alternatively, the proposition can be deduced from properties of the arrow polynomial of virtual links  \cite{virtualarrow}, which recall takes values in $\Z[a^{\pm 1},K_1,\ldots,K_n]$. Setting $a=1$, one obtains a polynomial invariant of virtual links which the skein relation implies is also invariant under crossing changes. Thus it is an invariant of flat links and coincides with the flat arrow polynomial defined above. 

Let $\pi$ denote the shadow map from virtual knots to flat knots. If $L$ is a virtual link with arrow polynomial  $f(L)\in\Z[a^{\pm 1},K_1,\ldots,K_n]$ and 
shadow image $\alpha=\pi(L)$, then $\bar A(\alpha)=f(L)\vert_{a=1}$.

Further, it is well-known that substituting $K_i=1$ into the arrow polynomial of $L$ gives the Jones polynomial of $L$. Further, for any virtual links, that setting $a=1$ in the Jones polynomial gives 1. Thus, for any flat knot $\alpha$, it follows that 
\begin{equation} \label{eqn:flatarroweqn}
\bar A(\alpha)\vert_{K_i=1}=1.    
\end{equation}

The following example shows that the flat arrow polynomial is not multiplicative under connected sum. 
For example, the flat knots 4.5 and 6.132  in Figure~\ref{fig:fk6132} are connected sums of two diagrams of the flat unknot. However, their flat arrow polynomials are 
\begin{align*}
\bar{A}(4.5) & = -4 K_1^2 + 2 K_2 + 3,\\
\bar{A}(6.132) & = -16K_1^4 + 8K_1^2K_2 + 8K_1^2 + 1.
\end{align*}
If the flat arrow polynomial were multiplicative, then we would have $\bar{A}(4.5)=\bar{A}(6.132)= 1.$ Since that is not the case, we conclude that $\bar{A}(\alpha)$ is not multiplicative under connected sum.

The flat knots 6.139 and 6.549 in Figure~\ref{fig:fk6132} 
can both be realized as $D\# D'$, where $D$ and $D'$ are minimal crossing diagrams  of $-$3.1 and 3.1, respectively. However, their flat arrow polynomials are 
\begin{align*}
    \bar{A}(6.139) & = 4 K_1^2 K_2 - 4 K_1 K_3 + K_4,\\
\bar{A}(6.549) & = 1.
\end{align*}
Observe that the flat arrow polynomial of $D$ is nontrivial, indeed $\bar{A}(3.1) = 2 K_1^2 - K_2.$ The same is true for $D'$.
However, their connected sum $6.549 =D\#D'$ has the trivial flat arrow polynomial.
These examples also show that the constant term of the flat arrow polynomial is not multiplicative under connected sum.

\begin{figure}[ht]
   \centering
   \includegraphics[width=0.9\textwidth]{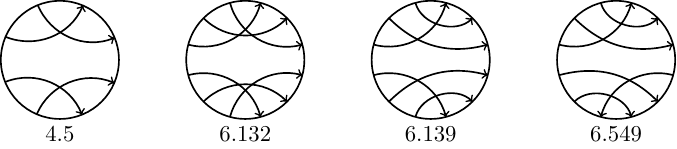}
   \caption{The flat arrow polynomial is not multiplicative under connected sum.}
   \label{fig:fk6132}
\end{figure}

\subsection{Free equivalence, checkerboard colorability, and almost classicality}
In this subsection, we relate checkerboard colorability and almost classicality of flat knots.
When the flat knot is almost classical, we show that its flat arrow polynomial is trivial, and when it is checkerboard colorable,
we show its flat arrow polynomial has an odd constant term.

We begin by defining the notion of free equivalence of flat knots.
\begin{definition}[\cite{turaev08}]
    Two flat knots are said to be \emph{free-equivalent} if they are related by the following relation.
    The set of \emph{free knots} consists of flat knots modulo free-equivalence.
\begin{figure}[ht]
    \centering
    \includegraphics[scale=1.0]{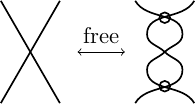}
    \hspace{22pt}
    \includegraphics[scale=1.0]{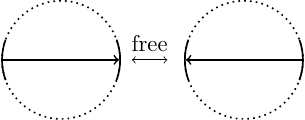}
    \caption[Free-equivalence]{Free-equivalence on a flat knot diagram and a flat Gauss diagram}
    \label{fig:freemove}
\end{figure}
\end{definition}

Note that all flat knots up to crossing number 4 are free-equivalent to the unknot.
In \cite{manturov-free}, Manturov defined a non-negative integer-valued
invariant $L$ of free knots (cf. \cite{MR4213072}*{Section~6}). It is an obstruction to the knot being \emph{freely slice}.
There are examples of free knots with nontrivial $L$-invariant. Figure~\ref{fig:fk_free} shows two flat knots with 5 crossings that are nontrivial as free knots. In fact, one can compute that $L(5.19)=L(5.36)=4$, and this implies that neither of them is freely slice.

\begin{figure}[ht]
   \centering
   \includegraphics[scale=1.0]{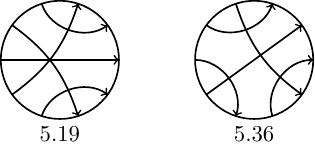}
   \caption{Flat knots that are not freely trivial}
   \label{fig:fk_free}
\end{figure}

\begin{definition}
A flat Gauss diagram is said to be of \emph{alternating pattern} if its underlying {\tt OU}-word is ``{\tt OUOU $\cdots$ OU}''.
\end{definition}
Note that this definition is adapted from the definition of alternating pattern for virtual knot in \cite{karimi18}.

\begin{lemma}[\cite{karimi18}]
    Every flat knot with alternating pattern is almost classical.
   \label{lem:karimi18}
\end{lemma}

\begin{lemma}
Any flat knot free-equivalent to a checkerboard colorable flat knot is 
also checkerboard colorable.
   \label{lem:ccfr}
\end{lemma}
\begin{proof}
    We consider the effect on the arrow index $n(f)$ of applying a free-equivalence to an arrow $e$. Suppose first that $e\neq f$. 
  Let $f^+$ be the arc on Gauss diagram from the tail to the head of the arrow $f$.
    If $f^+$ contains both the arrow tail and head of $e$, then the index $n(f)$ does not change.
    The same is true if neither the head nor the tail of $e$ is contained on $f^+$. 
    If $f^+$ contains one of them (e.g. the arrow head of $e$ but not the tail, or vice versa), then the index $n(f)$  changes by $\pm 2$.
    
    In the case $e=f$, it is easy to verify that the index $n(e)$ changes by sign.
    Therefore, the parity (even/odd) of all the arrows is preserved under free-equivalence, and whether a chord is even or odd is well-defined for free knots. In particular, 
    if two flat knot diagrams $D_1,D_2$ are free-equivalent and if all the arrows of $D_1$ are even (i.e.,  if $n(e)\equiv 0 \mod 2$ for all $e\in \arr(D_1)$), then the same must be true for $D_2.$
\end{proof}

\begin{lemma}
    A flat knot is almost classical if and only if it has a diagram (not necessarily minimal) of alternating pattern.
   \label{thm:acalt}
\end{lemma}
\begin{proof}
    An almost classical flat knot represents an immersed loop bounding an immersed oriented surface $F$ in a Carter surface. 
    At the cost of increasing the crossing number, we can turn $F$ to a disk attached by finitely many bands,
    where crossings only occur in quadruples when a band crosses another band. In this way, $F$ has only one side facing to the positive side of the Carter surface.
    By this construction, the crossings are alternating.
    Conversely, by Lemma~\ref{lem:karimi18},  a  Gauss diagram
    with alternating pattern ``{\tt OUOU $\cdots$ OU}'' has only chords of index zero and hence is almost classical.
\end{proof}

The alternating pattern cannot always be chosen to have minimal crossing number.
For example, the almost classical flat knots \rm{ac8.16} (also called \rm{8.1240457})
and \rm{ac10.1088}  in Figure~\ref{fig:fkac101097} 
do not have minimal crossing alternating pattern diagrams.
\begin{figure}[ht]
   \centering
   \includegraphics[scale=1.2]{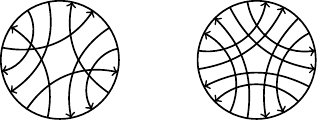}
      \caption{Almost classical flat knots 
          \rm{ac8.16}
      and \rm{ac10.1088}}
   \label{fig:fkac101097}
\end{figure}
\begin{lemma}
    A flat knot is checkerboard colorable 
    if and only if it has a Gauss diagram that is free-equivalent to a diagram of alternating pattern.
   \label{lem:ccalt}
\end{lemma}
\begin{proof}
    A flat checkerboard colorable diagram lifts to checkerboard colorable virtual diagrams.
    By \cite{kamada02}*{Lemma~7}, a checkerboard colorable virtual diagram can be made alternating by crossing changes, which does not change the flat knot type it projects to.
    Therefore, for every flat checkerboard colorable diagram, there exists an alternating virtual diagram $D$ projecting to it. Apply free-moves at all negative crossings of flat diagram $\pi(D)$, we obtain a flat diagram of alternating pattern.
\end{proof}

\begin{lemma}
    If the flat knot $\alpha$ is almost classical, then $A(\alpha)=1$.
    \label{lem:acarrow}
\end{lemma}
\begin{proof}
    We know $\alpha$ can be represented as the boundary of an immersed Seifert surface
    $ F\looparrowright \Sigma$ and we can alter the singular points so that we get an embedding  $F\rightarrow \Sigma\times I$, whose boundary represents a virtual knot, say $K$.
    By this construction, we have $K$ is also almost classical (also called null-homologous) $\pi(K)=\alpha$, where $\pi$ is the shadow projection.
By \cite{miller22}*{Theorem~3.21}, if the virtual knot $K$ is almost classical then its  arrow polynomial
is  the same as its Jones polynomial. 
Then by Equation \eqref{eqn:flatarroweqn}, $\bar A(\alpha)=1$.
\end{proof}

We use $C(D)$ to denote the constant term of $A(D)$ for a flat knot diagram $D$, 
and we use $\bar{C}(\alpha)$ for the constant term of $\bar{A}(\alpha)$ for a flat knot $\alpha$.
If $D$ is a diagram with $n$ crossings representing $\alpha$, then these are related by $\bar{C}(\alpha) = (-1)^n C(D)$.

\begin{theorem}
    If the flat knot $\alpha$ is checkerboard colorable, then $C(\alpha)\equiv 1 \mod 2$.
    \label{thm:oddconst}
\end{theorem}

\begin{proof}
    By Lemma~\ref{thm:acalt}, Lemma~\ref{lem:ccalt} and Lemma~\ref{lem:acarrow}, $\alpha$ is obtained from some almost classical Gauss diagram $D$ with $C(D)=\pm 1$ by free-move or (arrow-change in Gauss diagram).
    Observe that when we apply one free-move in Figure~\ref{fig:freemove},
the state resolution described in Figure~\ref{fig:skein}
does not change except that the cusps in (1) are in opposite directions.
If the state $S$ has more than two loops, then $(-1)^{n} (-2)^{|S|-1} \langle S \rangle $
either has zero constant term  or $(-2)^{|S|-1}$.
If the state $S$ has only one loop, then the two cusps  are in the same loop,
then either the number of cusps after deduction is either changed by at most $\pm 4$.
However, by \cite{miller22}*{Theorem~3.31}, the cusps number before
and after a free-move  can be only be $ 8n $: since before and after the move 
the flat knot remains checkerboard colorable and thus they both lift to some 
checkerboard colorable virtual knots which has only $K_{4n}$ for single-loop states mapping 
to monomials in their  arrow polynomials.
Therefore, the constant term of   an checkerboard colorable knot remains an odd number.
\end{proof}

\begin{figure}[ht]
    \centering
   \includegraphics[scale=1.0]{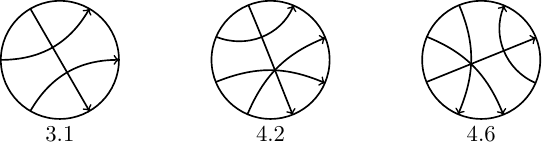}
    \caption{Three non-checkerboard-colorable flat knots}
    \label{fig:fk42}
\end{figure}
The flat knots shown in Figure~\ref{fig:fk42}
have
$\bar C(3.1)=0$,
$\bar C(4.2)=2$,
$\bar C(4.6)=2$.
Since these are all even, Theorem \ref{thm:oddconst} applies to show that these flat knots are not checkerboard colorable.

\begin{remark}
    For all flat knots up to $7$ crossings, one can check that $\bar{C}(\alpha)$ is odd whenever $\alpha$ is slice (refer to Definition~\ref{def:sliceflat}.
 for the definition of sliceness).
However, there exist flat knots $\beta$ that are slice such that $\bar{C}(\beta)$ is even. For example, the flat knot 8.11946 in Figure~\ref{fig:fk811946} is slice and has
$\bar A(8.11946)=12 K_1^3 + 4 K_1^2 K_2 + 4 K_1^2 K_3 - 12 K_1^2 + 4 K_1 K_2^2 - 20 K_1 K_2 - 4 K_1 K_3 - 4 K_2 K_3 + 6 K_2 + 4 K_3 + K_4 + 6$ and hence $\bar C(8.11946)=6$.
    \label{rmk:constant}
\end{remark}

\begin{figure}[ht]
   \centering
   \includegraphics[scale=1.1]{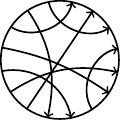}
   \;
   \;
   \;
   \;
   \;
   \includegraphics[scale=1.1]{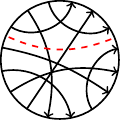}
   \caption{Flat knot 8.11946 and saddle move to slice it}
   \label{fig:fk811946}
\end{figure}
\subsection{Distinguishing flat knots}

Given a flat knot diagram $D$, let $D^n$ denote its $n$-strand cable. Thus $D^n$ is the flat link diagram with $n$ components obtained from taking the $n$-fold parallels at each virtual and flat crossing as in Figure \ref{fig:cable} and connecting them up without introducing any additional flat or virtual crossings. Given an orientation on $D$, we orient all the components of $D^n$ in the same direction. It is not difficult to show that the flat link type of $D^n$, as an oriented link, depends only on the oriented flat knot type of $D$. Given a flat knot $\alpha,$ we use $\alpha^n$ to denote its $n$-strand cable of $\alpha$, which  is well-defined and independent of the flat diagram used to represent $\alpha$.

\begin{figure}[ht]
   \centering
   \includegraphics[scale=1.1]{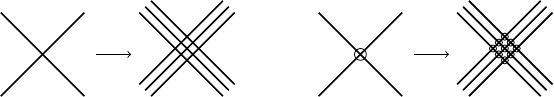}
   \caption{3-strand cable at flat and virtual crossings}
   \label{fig:cable}
\end{figure}

As with classical knots, link invariants of the cables $\alpha^n$ are invariants of the underlying flat knot $\alpha$. One can often obtain more powerful invariants at the expense of computability. Indeed, as we shall see, combining the flat arrow polynomial with cabling leads to stronger invariants that are more effective at distinguishing flat knots. In fact, using $2$-strand cabled arrow polynomials alone, we can distinguish flat knots up to 4 crossings completely.
Combined with the $\phi$-invariant, we can further distinguish flat knots completely up to 6 crossings and with just six pairs of 7-crossing knots not separated, see  Table~\ref{table:arr_dis}.\footnote{  When $cr(\alpha)\ge 7$, we only calculated $A(\alpha^2),C(\alpha^3)$ if they are not distinguished by other invariants.  Some calculation of $C(\alpha^3)$ of some $7$-crossing flat knots are not finished due to the workload of the calculation.}

\begin{table}[ht]
    \centering
\begin{tabular}{||c |c |c|c|  c| c||} 
 \hline
 Crossings &  \# Flat knots  &   $\phi$ & $A(\alpha)$ & $A(\alpha^2)$ & $A(\alpha^2)$, $C(\alpha^3)$  and $\phi$ \\ [0.5ex] 
 \hline\hline
3 & 1 & 0 & 0& 0 &0\\
 \hline
4 & 11 & 0& 10& 0 &0\\
 \hline
5 & 120 & 8&111& 2  &0\\
 \hline
6 & 2086 & 74 &1919& 10 &0\\
 \hline
 7 & 46233 & 1375& 42163&--- &12\\
 \hline
\end{tabular}
\vspace{2mm}
\caption{Number of non-distinguished flat knots using the invariant(s)}
\label{table:arr_dis}
\end{table} 

Specifically, there are two $4$-crossing flat knots, eight $5$-crossing flat knots, $106$ $6$-crossing flat knots, and $674$ $7$-crossing flat knots that are not distinguished from the unknot by $A(\alpha)$.

The first pair of flat knots that cannot be distinguished by the $2$-strand cabled arrow polynomial are 5.112 and 5.113. They appear in Figure \ref{fig:fk5112} and have
$$
    A(5.112^2)=A(5.112^2)  =-64K_1^4 + 144 K_1^2 K_2 - 80 K_1^2 - 56 K_2^2 + 54
$$

\begin{figure}[ht]
   \centering
   \includegraphics[scale=1.0]{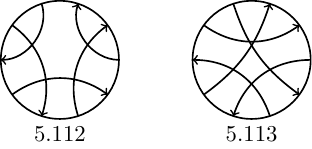}
   \caption{Flat knots with identical 2-strand cabled arrow polynomial.}
   \label{fig:fk5112}
\end{figure}

We know that the arrow polynomial and cabled arrow polynomials of any almost classical flat knot are trivial. To distinguish almost classical knots, we rely on the primitive based matrices.
However, there exist nontrivial flat knots with trivial primitive based matrices. 
\begin{example} \label{exa:ac8-19}
Consider the flat knot \rm{ac8.19} in Figure \ref{fig:fkac828}. It has trivial primitive based matrix. In fact, every invariant of flat knots studied up to now is trivial for the flat knot \rm{ac8.19}. (This flat knot is also called 8.1241248, and it 
is the $19$-th 8-crossing almost classical knot and the $1241248$-th 8-crossing flat knot. One can find the knot on \cite{flatknotinfo} by its name 8.1241248.) At this point, the monotonicity algorithm is the only way to deduce that \rm{ac8.19} is nontrivial as a flat knot. 
\end{example}

\begin{figure}[ht]
   \centering
   \includegraphics[scale=1.1]{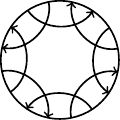}
   \caption[Almost classical flat knot \rm{ac8.19}]{Almost classical flat knot \rm{ac8.19} has trivial primitive based matrix.}
   \label{fig:fkac828}
\end{figure}

Figure~\ref{fig:ac10} shows further examples of almost classical flat knots with trivial primitive based matrices. They all have 10 crossings.  
With the invariants we have introduced so far, it is impossible to show these examples are nontrivial.
In the next section, we will introduce new invariants of flat knots which are powerful enough to show nontriviality for \rm{ac8.19} and the examples in Figure~\ref{fig:ac10}.

\begin{figure}[ht]
    \centering
    \includegraphics[angle=0,scale=1.0,trim=0 0 0 0]{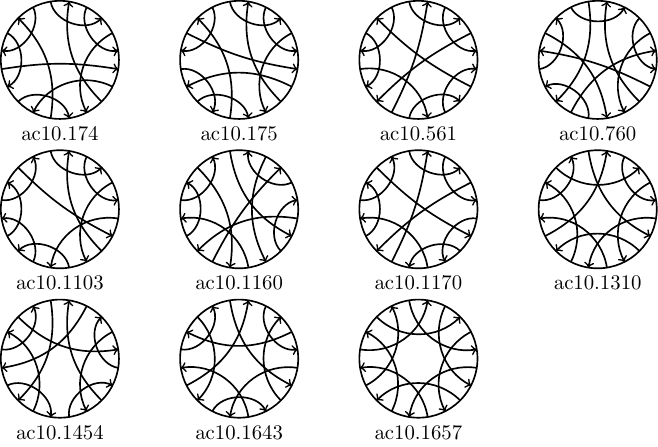}
\caption{Almost classical knots with trivial primitive based matrix}
    \label{fig:ac10}
\end{figure}

\medskip

We conclude this section with a few open problems.
\begin{problem}
    Is there a bigraded invariant that categorifies the flat arrow polynomial, cf.~\cite{category}? 
    \label{prob:cate_arr}
\end{problem}

\begin{problem}
    Can one use the flat arrow polynomial to extract slice obstructions for flat knots? 
    \label{prob:slice_arr}
\end{problem}

\begin{problem}
    Which polynomials can be realized as flat arrow polynomials of flat knots? 
    \label{prob:rel_arr}
\end{problem}

\section{The flat Jones-Krushkal polynomial}
\label{chap:jk}
In this section, we define the flat Jones-Krushkal polynomial,  as well as the normalization and enhancement of it. 
We apply them to the problem of distinguishing flat knots. 
The calculations in this section were performed using a matrix-based algorithm; and we refer the reader to \cite{chen-thesis}*{Section~5.2 \& Appendix C} for details on the algorithm and accompanying Python code.
\subsection{Definition}
We begin by introducing the flat Jones-Krushkal polynomial, defined for flat knot diagrams on closed surfaces.
It is closely related to the homological Jones polynomial for links in thickened surfaces, which 
was introduced by Krushkal in \cite{krushkal} and further studied in \cite{jkpoly}.

    Let $D$ be a flat knot diagram with $n$ crossings on a closed surface $\Sigma$.
By applying the skein relation of Figure~\ref{fig:skein_jk}(1) to each flat crossing of $D$, we obtain  
$2^{n}$ states. The states can be indexed by a map $\left\{  1,\ldots,n\right\} \rightarrow \left\{ 0,1 \right\}$.
Denote this set of states by $\mathfrak{S}$.

\begin{figure}[ht]
    \centering
\includegraphics[angle=0,scale=3.2,trim=0 0 0 0]{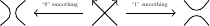}
\caption{Two types of  smoothings}
    \label{fig:skein_jk}
\end{figure}

Each state $S\in \mathfrak{S}$ contains simple closed loops embedded in $\Sigma$. The embedding induces a map $i_*:H_1(S;\Z/2) \rightarrow H_1(\Sigma_g;\Z/2)$.
\begin{definition}
     The \emph{homological Kauffman bracket} of state $S$ is denoted by
     $\langle D\, |\, S\rangle_\Sigma$ and given by $$\langle D\, |\, S\rangle_\Sigma=(-2)^{k(S)} z^{r(S)},$$ where
    $$ k(S) = \dim (\ker (i_*:H_1(S;\Z/2) \rightarrow H_1(\Sigma_g;\Z/2))), $$
    $$ r(S) = \dim (\im (i_*:H_1(S;\Z/2) \rightarrow H_1(\Sigma_g;\Z/2))). $$
Given a flat knot diagram $D$ on the closed surface $\Sigma$, the \emph{flat Jones-Krushkal polynomial} is defined by 
$$\ds J_D(z)= (-1)^{cr(D)} \sum_{S\in \mathfrak{S} } \langle D\, |\, S \rangle_\Sigma= (-1)^{cr(D)}\sum_{S\in \mathfrak{S} } (-2)^{k(S)} z^{r(S)}.$$ 
\end{definition}

\begin{proposition}
    If     $D$ is checkerboard colorable, then $2|J_{D}$.
    If     $D$ is not checkerboard colorable, then $z|J_{D}$.
\label{thm:facotrjk}
\end{proposition}
\begin{proof}
    Let $\omega_D(S^1)$ be the flat knot diagram on $\Sigma$.
    Then the sum of loops in each state $S$ of $\omega_D(S^1)$ is equal to  $[ \omega_D(S^1) ]\in H_1(\Sigma;\Z/2)$.
    Therefore, if  $D$ is checkerboard colorable, then $k(S)>0$ for each state.
    If  $D$ is not checkerboard colorable, then $r(S)>0$ for each state.
\end{proof}

Note that when $z|J_{D}$, the flat Jones-Krushkal polynomial has zero constant term, i.e., $J_D(0)=0$.

Now we show  $J_D(z)$ is an invariant under  FR3 moves in Figure~\ref{fig:frmove} for flat diagrams on the same surface.
\begin{proposition}
    Let $D,D'$ be two flat diagram on surface $\Sigma$, if  $D,D'$ are related by one 
    FR3 move, then $J_D(z)= J_{ D' }(z)$.
    \label{prop:jkfr3}
\end{proposition}

\begin{proof}
    This follows from the calculation of the flat skein bracket:

\begin{eqnarray*}
\langle \mbox{\begin{tikzpicture}[scale=0.24,rotate=0]
        \draw[line width=0.35mm,-,yshift=-0mm] (50:1.1cm) to (-130:1.1cm);
        \draw[line width=0.35mm,-] (0:1.2cm) to [out=140,in=40] (180:1.2cm);
        \draw[line width=0.35mm,-,yshift=-0mm] (-50:1.1cm) to (130:1.1cm);
\end{tikzpicture}} \rangle_\Sigma
&=&
\langle \mbox{\begin{tikzpicture}[scale=0.24]
            \draw[line width=0.35mm,rotate=30] (30:5mm) to   (17:9mm);
            \draw[line width=0.35mm,rotate=30] (-30:5mm) to   (-17:9mm);
            \draw[line width=0.35mm,rotate=150] (30:5mm) to  (17:9mm);
            \draw[line width=0.35mm,rotate=150] (-30:5mm) to  (-17:9mm);
            \draw[line width=0.35mm,rotate=270] (30:5mm)  to  (17:9mm);
            \draw[line width=0.35mm,rotate=270] (-30:5mm) to  (-17:9mm);
            \draw[line width=0.35mm,rotate=0] (60:5mm) to (120:5mm);
            \draw[line width=0.35mm,rotate=120] (60:5mm) to (120:5mm);
            \draw[line width=0.35mm,rotate=-120] (60:5mm) to (120:5mm);
\end{tikzpicture}}\rangle_\Sigma
+
\langle \mbox{\begin{tikzpicture}[scale=0.24]
            \draw[line width=0.35mm,rotate=30] (30:5mm) to  (17:9mm);
            \draw[line width=0.35mm,rotate=30] (-30:5mm) to  (-17:9mm);
            \draw[line width=0.35mm,rotate=150] (30:5mm) to  (17:9mm);
            \draw[line width=0.35mm,rotate=150] (-30:5mm) to  (-17:9mm);
            \draw[line width=0.35mm,rotate=270] (30:5mm) to[out=0,in=0]  (-30:5mm);
            \draw[line width=0.35mm,rotate=270] [xshift=5mm](30:5mm) to[out=180,in=180]  (-30:5mm);
            \draw[line width=0.35mm,rotate=0] (60:5mm) to (120:5mm);
            \draw[line width=0.35mm,rotate=120] (60:5mm) to (120:5mm);
            \draw[line width=0.35mm,rotate=-120] (60:5mm) to (120:5mm);
\end{tikzpicture}}\rangle_\Sigma
+
\langle \mbox{\begin{tikzpicture}[scale=0.24]
            \draw[line width=0.35mm,rotate=30] (30:5mm) to[out=0,in=0] (-30:5mm);
            \draw[line width=0.35mm,rotate=30] [xshift=5mm](30:5mm) to[out=180,in=180] (-30:5mm);
            \draw[line width=0.35mm,rotate=150] (30:5mm) to[out=0,in=0] (-30:5mm);
            \draw[line width=0.35mm,rotate=150] [xshift=5mm](30:5mm) to[out=180,in=180] (-30:5mm);
            \draw[line width=0.35mm,rotate=270] (30:5mm)  to  (17:9mm);
            \draw[line width=0.35mm,rotate=270] (-30:5mm) to  (-17:9mm);
            \draw[line width=0.35mm,rotate=0] (60:5mm) to (120:5mm);
            \draw[line width=0.35mm,rotate=120] (60:5mm) to (120:5mm);
            \draw[line width=0.35mm,rotate=-120] (60:5mm) to (120:5mm);
\end{tikzpicture}}\rangle_\Sigma
+
\langle \mbox{\begin{tikzpicture}[scale=0.24]
            \draw[line width=0.35mm,rotate=30] (30:5mm) to[out=0,in=0] (-30:5mm);
            \draw[line width=0.35mm,rotate=30] [xshift=5mm](30:5mm) to[out=180,in=180] (-30:5mm);
            \draw[line width=0.35mm,rotate=150] (30:5mm) to[out=0,in=0]   (-30:5mm);
            \draw[line width=0.35mm,rotate=150] [xshift=5mm](30:5mm) to[out=180,in=180]  (-30:5mm);
            \draw[line width=0.35mm,rotate=270] (30:5mm) to[out=0,in=0]  (-30:5mm);
            \draw[line width=0.35mm,rotate=270] [xshift=5mm](30:5mm) to[out=180,in=180]  (-30:5mm);
            \draw[line width=0.35mm,rotate=0] (60:5mm) to (120:5mm);
            \draw[line width=0.35mm,rotate=120] (60:5mm) to (120:5mm);
            \draw[line width=0.35mm,rotate=-120] (60:5mm) to (120:5mm);
\end{tikzpicture}}\rangle_\Sigma 
+
\langle \mbox{\begin{tikzpicture}[scale=0.24,rotate=-120]
            \draw[line width=0.35mm,rotate=30] (30:5mm) to   (17:9mm);
            \draw[line width=0.35mm,rotate=30] (-30:5mm) to   (-17:9mm);
            \draw[line width=0.35mm,rotate=150] (30:5mm) to (17:9mm);
            \draw[line width=0.35mm,rotate=150] (-30:5mm) to (-17:9mm);
            \draw[line width=0.35mm,rotate=270] (30:5mm) to[out=0,in=0]  (-30:5mm);
            \draw[line width=0.35mm,rotate=270] [xshift=5mm](30:5mm) to[out=180,in=180]  (-30:5mm);
            \draw[line width=0.35mm,rotate=0] (60:5mm) to (120:5mm);
            \draw[line width=0.35mm,rotate=120] (60:5mm) to (120:5mm);
            \draw[line width=0.35mm,rotate=-120] (60:5mm) to (120:5mm);
\end{tikzpicture}}\rangle_\Sigma
+
\langle \mbox{\begin{tikzpicture}[scale=0.24,rotate=-120]
            \draw[line width=0.35mm,rotate=30] (30:5mm) to[out=0,in=0] (-30:5mm);
            \draw[line width=0.35mm,rotate=30] [xshift=5mm](30:5mm) to[out=180,in=180] (-30:5mm);
            \draw[line width=0.35mm,rotate=150] (30:5mm) to[out=0,in=0] (-30:5mm);
            \draw[line width=0.35mm,rotate=150] [xshift=5mm](30:5mm) to[out=180,in=180] (-30:5mm);
            \draw[line width=0.35mm,rotate=270] (30:5mm)  to  (17:9mm);
            \draw[line width=0.35mm,rotate=270] (-30:5mm) to  (-17:9mm);
            \draw[line width=0.35mm,rotate=0] (60:5mm) to (120:5mm);
            \draw[line width=0.35mm,rotate=120] (60:5mm) to (120:5mm);
            \draw[line width=0.35mm,rotate=-120] (60:5mm) to (120:5mm);
\end{tikzpicture}}\rangle_\Sigma
+
\langle \mbox{\begin{tikzpicture}[scale=0.24,rotate=120]
            \draw[line width=0.35mm,rotate=30] (30:5mm) to[out=0,in=0] (-30:5mm);
            \draw[line width=0.35mm,rotate=30] [xshift=5mm](30:5mm) to[out=180,in=180] (-30:5mm);
            \draw[line width=0.35mm,rotate=150] (30:5mm) to[out=0,in=0] (-30:5mm);
            \draw[line width=0.35mm,rotate=150] [xshift=5mm](30:5mm) to[out=180,in=180] (-30:5mm);
            \draw[line width=0.35mm,rotate=270] (30:5mm)  to (17:9mm);
            \draw[line width=0.35mm,rotate=270] (-30:5mm) to (-17:9mm);
            \draw[line width=0.35mm,rotate=0] (60:5mm) to (120:5mm);
            \draw[line width=0.35mm,rotate=120] (60:5mm) to (120:5mm);
            \draw[line width=0.35mm,rotate=-120] (60:5mm) to (120:5mm);
\end{tikzpicture}}\rangle_\Sigma
+\langle \mbox{\begin{tikzpicture}[scale=0.24,rotate=120]
            \draw[line width=0.35mm,rotate=30] (30:5mm) to   (17:9mm);
            \draw[line width=0.35mm,rotate=30] (-30:5mm) to   (-17:9mm);
            \draw[line width=0.35mm,rotate=150] (30:5mm) to  (17:9mm);
            \draw[line width=0.35mm,rotate=150] (-30:5mm) to  (-17:9mm);
            \draw[line width=0.35mm,rotate=270] (30:5mm) to[out=0,in=0]  (-30:5mm);
            \draw[line width=0.35mm,rotate=270] [xshift=5mm](30:5mm) to[out=180,in=180]  (-30:5mm);
            \draw[line width=0.35mm,rotate=0] (60:5mm) to (120:5mm);
            \draw[line width=0.35mm,rotate=120] (60:5mm) to (120:5mm);
            \draw[line width=0.35mm,rotate=-120] (60:5mm) to (120:5mm);
\end{tikzpicture}}\rangle_\Sigma
\\
&=&
\langle \mbox{\begin{tikzpicture}[scale=0.24]
            \draw[line width=0.35mm,rotate=30] (30:5mm)   to   (17:9mm);
            \draw[line width=0.35mm,rotate=30] (-30:5mm)  to   (-17:9mm);
            \draw[line width=0.35mm,rotate=150] (30:5mm)  to   (17:9mm);
            \draw[line width=0.35mm,rotate=150] (-30:5mm) to   (-17:9mm);
            \draw[line width=0.35mm,rotate=270] (30:5mm)  to   (17:9mm);
            \draw[line width=0.35mm,rotate=270] (-30:5mm) to   (-17:9mm);
            \draw[line width=0.35mm,rotate=0] (60:5mm) to (120:5mm);
            \draw[line width=0.35mm,rotate=120] (60:5mm) to (120:5mm);
            \draw[line width=0.35mm,rotate=-120] (60:5mm) to (120:5mm);
\end{tikzpicture}}\rangle_\Sigma
+
\langle \mbox{\begin{tikzpicture}[scale=0.20,rotate=-60,xscale=-1]
            \draw[line width=0.35mm] (60:9mm) to   (0:9mm);
            \draw[line width=0.35mm] (120:9mm) to (-60:9mm);
            \draw[line width=0.35mm] (180:9mm) to   (240:9mm);
\end{tikzpicture}}\rangle_\Sigma
+
\langle \mbox{\begin{tikzpicture}[scale=0.24,rotate=60]
            \draw[line width=0.35mm,rotate=30] (30:5mm)   to   (17:9mm);
            \draw[line width=0.35mm,rotate=30] (-30:5mm)  to   (-17:9mm);
            \draw[line width=0.35mm,rotate=150] (30:5mm)  to   (17:9mm);
            \draw[line width=0.35mm,rotate=150] (-30:5mm) to   (-17:9mm);
            \draw[line width=0.35mm,rotate=270] (30:5mm)  to   (17:9mm);
            \draw[line width=0.35mm,rotate=270] (-30:5mm) to   (-17:9mm);
            \draw[line width=0.35mm,rotate=0] (60:5mm) to (120:5mm);
            \draw[line width=0.35mm,rotate=120] (60:5mm) to (120:5mm);
            \draw[line width=0.35mm,rotate=-120] (60:5mm) to (120:5mm);
\end{tikzpicture}}\rangle_\Sigma
-2\langle \mbox{\begin{tikzpicture}[scale=0.24,rotate=60]
            \draw[line width=0.35mm,rotate=30] (30:5mm)   to   (17:9mm);
            \draw[line width=0.35mm,rotate=30] (-30:5mm)  to   (-17:9mm);
            \draw[line width=0.35mm,rotate=150] (30:5mm)  to   (17:9mm);
            \draw[line width=0.35mm,rotate=150] (-30:5mm) to   (-17:9mm);
            \draw[line width=0.35mm,rotate=270] (30:5mm)  to   (17:9mm);
            \draw[line width=0.35mm,rotate=270] (-30:5mm) to   (-17:9mm);
            \draw[line width=0.35mm,rotate=0] (60:5mm) to (120:5mm);
            \draw[line width=0.35mm,rotate=120] (60:5mm) to (120:5mm);
            \draw[line width=0.35mm,rotate=-120] (60:5mm) to (120:5mm);
\end{tikzpicture}}\rangle_\Sigma
+\langle \mbox{\begin{tikzpicture}[scale=0.20,rotate=-60]
            \draw[line width=0.35mm] (60:9mm) to (0:9mm);
            \draw[line width=0.35mm] (120:9mm) to (-60:9mm);
            \draw[line width=0.35mm](180:9mm) to (240:9mm);
\end{tikzpicture}}\rangle_\Sigma
+
\langle \mbox{\begin{tikzpicture}[scale=0.24,rotate=60]
            \draw[line width=0.35mm,rotate=30] (30:5mm)   to   (17:9mm);
            \draw[line width=0.35mm,rotate=30] (-30:5mm)  to   (-17:9mm);
            \draw[line width=0.35mm,rotate=150] (30:5mm)  to   (17:9mm);
            \draw[line width=0.35mm,rotate=150] (-30:5mm) to   (-17:9mm);
            \draw[line width=0.35mm,rotate=270] (30:5mm)  to   (17:9mm);
            \draw[line width=0.35mm,rotate=270] (-30:5mm) to   (-17:9mm);
            \draw[line width=0.35mm,rotate=0] (60:5mm) to (120:5mm);
            \draw[line width=0.35mm,rotate=120] (60:5mm) to (120:5mm);
            \draw[line width=0.35mm,rotate=-120] (60:5mm) to (120:5mm);
\end{tikzpicture}}\rangle_\Sigma
+
\langle \mbox{\begin{tikzpicture}[scale=0.24,rotate=60]
            \draw[line width=0.35mm,rotate=30] (30:5mm)   to   (17:9mm);
            \draw[line width=0.35mm,rotate=30] (-30:5mm)  to   (-17:9mm);
            \draw[line width=0.35mm,rotate=150] (30:5mm)  to   (17:9mm);
            \draw[line width=0.35mm,rotate=150] (-30:5mm) to   (-17:9mm);
            \draw[line width=0.35mm,rotate=270] (30:5mm)  to   (17:9mm);
            \draw[line width=0.35mm,rotate=270] (-30:5mm) to   (-17:9mm);
            \draw[line width=0.35mm,rotate=0] (60:5mm) to (120:5mm);
            \draw[line width=0.35mm,rotate=120] (60:5mm) to (120:5mm);
            \draw[line width=0.35mm,rotate=-120] (60:5mm) to (120:5mm);
\end{tikzpicture}}\rangle_\Sigma
+
\langle \mbox{\begin{tikzpicture}[scale=0.20,rotate=0]
            \draw[line width=0.35mm] (60:9mm) to (0:9mm);
            \draw[line width=0.35mm] (120:9mm) to (-60:9mm);
            \draw[line width=0.35mm] (180:9mm) to (240:9mm);
\end{tikzpicture}}\rangle_\Sigma
\\
&=&
\langle \mbox{\begin{tikzpicture}[scale=0.24]
            \draw[line width=0.35mm,rotate=30] (30:5mm)   to   (17:9mm);
            \draw[line width=0.35mm,rotate=30] (-30:5mm)  to   (-17:9mm);
            \draw[line width=0.35mm,rotate=150] (30:5mm)  to   (17:9mm);
            \draw[line width=0.35mm,rotate=150] (-30:5mm) to   (-17:9mm);
            \draw[line width=0.35mm,rotate=270] (30:5mm)  to   (17:9mm);
            \draw[line width=0.35mm,rotate=270] (-30:5mm) to   (-17:9mm);
            \draw[line width=0.35mm,rotate=0] (60:5mm) to (120:5mm);
            \draw[line width=0.35mm,rotate=120] (60:5mm) to (120:5mm);
            \draw[line width=0.35mm,rotate=-120] (60:5mm) to (120:5mm);
\end{tikzpicture}}\rangle_\Sigma
+
\langle \mbox{\begin{tikzpicture}[scale=0.20,rotate=-60,xscale=-1]
            \draw[line width=0.35mm] (60:9mm) to   (0:9mm);
            \draw[line width=0.35mm] (120:9mm) to (-60:9mm);
            \draw[line width=0.35mm] (180:9mm) to   (240:9mm);
\end{tikzpicture}}\rangle_\Sigma
+\langle \mbox{\begin{tikzpicture}[scale=0.20,rotate=-60]
            \draw[line width=0.35mm] (60:9mm) to (0:9mm);
            \draw[line width=0.35mm] (120:9mm) to (-60:9mm);
            \draw[line width=0.35mm] (180:9mm) to (240:9mm);
\end{tikzpicture}}\rangle_\Sigma
+
\langle \mbox{\begin{tikzpicture}[scale=0.24,rotate=60]
            \draw[line width=0.35mm,rotate=30] (30:5mm)   to   (17:9mm);
            \draw[line width=0.35mm,rotate=30] (-30:5mm)  to   (-17:9mm);
            \draw[line width=0.35mm,rotate=150] (30:5mm)  to   (17:9mm);
            \draw[line width=0.35mm,rotate=150] (-30:5mm) to   (-17:9mm);
            \draw[line width=0.35mm,rotate=270] (30:5mm)  to   (17:9mm);
            \draw[line width=0.35mm,rotate=270] (-30:5mm) to   (-17:9mm);
            \draw[line width=0.35mm,rotate=0] (60:5mm) to (120:5mm);
            \draw[line width=0.35mm,rotate=120] (60:5mm) to (120:5mm);
            \draw[line width=0.35mm,rotate=-120] (60:5mm) to (120:5mm);
\end{tikzpicture}}\rangle_\Sigma
+
\langle \mbox{\begin{tikzpicture}[scale=0.20,rotate=0]
            \draw[line width=0.35mm] (60:9mm) to (0:9mm);
            \draw[line width=0.35mm] (120:9mm) to (-60:9mm);
            \draw[line width=0.35mm] (180:9mm) to (240:9mm);
\end{tikzpicture}}\rangle_\Sigma
\\
&=&
\langle \mbox{\begin{tikzpicture}[yscale=-1,scale=0.24]
            \draw[line width=0.35mm,rotate=30] (30:5mm)   to   (17:9mm);
            \draw[line width=0.35mm,rotate=30] (-30:5mm)  to   (-17:9mm);
            \draw[line width=0.35mm,rotate=150] (30:5mm)  to   (17:9mm);
            \draw[line width=0.35mm,rotate=150] (-30:5mm) to   (-17:9mm);
            \draw[line width=0.35mm,rotate=270] (30:5mm)  to   (17:9mm);
            \draw[line width=0.35mm,rotate=270] (-30:5mm) to   (-17:9mm);
            \draw[line width=0.35mm,rotate=0] (60:5mm) to (120:5mm);
            \draw[line width=0.35mm,rotate=120] (60:5mm) to (120:5mm);
            \draw[line width=0.35mm,rotate=-120] (60:5mm) to (120:5mm);
\end{tikzpicture}}\rangle_\Sigma
+
\langle \mbox{\begin{tikzpicture}[yscale=-1,scale=0.20,rotate=-60,xscale=-1]
            \draw[line width=0.35mm] (60:9mm) to   (0:9mm);
            \draw[line width=0.35mm] (120:9mm) to (-60:9mm);
            \draw[line width=0.35mm] (180:9mm) to   (240:9mm);
\end{tikzpicture}}\rangle_\Sigma
+
\langle \mbox{\begin{tikzpicture}[yscale=-1,scale=0.24,rotate=60]
            \draw[line width=0.35mm,rotate=30] (30:5mm)   to   (17:9mm);
            \draw[line width=0.35mm,rotate=30] (-30:5mm)  to   (-17:9mm);
            \draw[line width=0.35mm,rotate=150] (30:5mm)  to   (17:9mm);
            \draw[line width=0.35mm,rotate=150] (-30:5mm) to   (-17:9mm);
            \draw[line width=0.35mm,rotate=270] (30:5mm)  to   (17:9mm);
            \draw[line width=0.35mm,rotate=270] (-30:5mm) to   (-17:9mm);
            \draw[line width=0.35mm,rotate=0] (60:5mm) to (120:5mm);
            \draw[line width=0.35mm,rotate=120] (60:5mm) to (120:5mm);
            \draw[line width=0.35mm,rotate=-120] (60:5mm) to (120:5mm);
\end{tikzpicture}}\rangle_\Sigma
-2\langle \mbox{\begin{tikzpicture}[yscale=-1,scale=0.24,rotate=60]
            \draw[line width=0.35mm,rotate=30] (30:5mm)   to   (17:9mm);
            \draw[line width=0.35mm,rotate=30] (-30:5mm)  to   (-17:9mm);
            \draw[line width=0.35mm,rotate=150] (30:5mm)  to   (17:9mm);
            \draw[line width=0.35mm,rotate=150] (-30:5mm) to   (-17:9mm);
            \draw[line width=0.35mm,rotate=270] (30:5mm)  to   (17:9mm);
            \draw[line width=0.35mm,rotate=270] (-30:5mm) to   (-17:9mm);
            \draw[line width=0.35mm,rotate=0] (60:5mm) to (120:5mm);
            \draw[line width=0.35mm,rotate=120] (60:5mm) to (120:5mm);
            \draw[line width=0.35mm,rotate=-120] (60:5mm) to (120:5mm);
\end{tikzpicture}}\rangle_\Sigma
+\langle \mbox{\begin{tikzpicture}[yscale=-1,scale=0.20,rotate=-60]
            \draw[line width=0.35mm] (60:9mm) to (0:9mm);
            \draw[line width=0.35mm] (120:9mm) to (-60:9mm);
            \draw[line width=0.35mm] (180:9mm) to (240:9mm);
\end{tikzpicture}}\rangle_\Sigma
+
\langle \mbox{\begin{tikzpicture}[yscale=-1,scale=0.24,rotate=60]
            \draw[line width=0.35mm,rotate=30] (30:5mm)   to   (17:9mm);
            \draw[line width=0.35mm,rotate=30] (-30:5mm)  to   (-17:9mm);
            \draw[line width=0.35mm,rotate=150] (30:5mm)  to   (17:9mm);
            \draw[line width=0.35mm,rotate=150] (-30:5mm) to   (-17:9mm);
            \draw[line width=0.35mm,rotate=270] (30:5mm)  to   (17:9mm);
            \draw[line width=0.35mm,rotate=270] (-30:5mm) to   (-17:9mm);
            \draw[line width=0.35mm,rotate=0] (60:5mm) to (120:5mm);
            \draw[line width=0.35mm,rotate=120] (60:5mm) to (120:5mm);
            \draw[line width=0.35mm,rotate=-120] (60:5mm) to (120:5mm);
\end{tikzpicture}}\rangle_\Sigma
+
\langle \mbox{\begin{tikzpicture}[yscale=-1,scale=0.24,rotate=60]
            \draw[line width=0.35mm,rotate=30] (30:5mm)   to   (17:9mm);
            \draw[line width=0.35mm,rotate=30] (-30:5mm)  to   (-17:9mm);
            \draw[line width=0.35mm,rotate=150] (30:5mm)  to   (17:9mm);
            \draw[line width=0.35mm,rotate=150] (-30:5mm) to   (-17:9mm);
            \draw[line width=0.35mm,rotate=270] (30:5mm)  to   (17:9mm);
            \draw[line width=0.35mm,rotate=270] (-30:5mm) to   (-17:9mm);
            \draw[line width=0.35mm,rotate=0] (60:5mm) to (120:5mm);
            \draw[line width=0.35mm,rotate=120] (60:5mm) to (120:5mm);
            \draw[line width=0.35mm,rotate=-120] (60:5mm) to (120:5mm);
\end{tikzpicture}}\rangle_\Sigma
+
\langle \mbox{\begin{tikzpicture}[yscale=-1,scale=0.20,rotate=0]
            \draw[line width=0.35mm] (60:9mm) to (0:9mm);
            \draw[line width=0.35mm] (120:9mm) to (-60:9mm);
            \draw[line width=0.35mm] (180:9mm) to (240:9mm);
\end{tikzpicture}}\rangle_\Sigma
\\
&=&
\langle \mbox{\begin{tikzpicture}[yscale=-1,scale=0.24]
            \draw[line width=0.35mm,rotate=30] (30:5mm) to   (17:9mm);
            \draw[line width=0.35mm,rotate=30] (-30:5mm) to   (-17:9mm);
            \draw[line width=0.35mm,rotate=150] (30:5mm) to  (17:9mm);
            \draw[line width=0.35mm,rotate=150] (-30:5mm) to  (-17:9mm);
            \draw[line width=0.35mm,rotate=270] (30:5mm)  to  (17:9mm);
            \draw[line width=0.35mm,rotate=270] (-30:5mm) to  (-17:9mm);
            \draw[line width=0.35mm,rotate=0] (60:5mm) to (120:5mm);
            \draw[line width=0.35mm,rotate=120] (60:5mm) to (120:5mm);
            \draw[line width=0.35mm,rotate=-120] (60:5mm) to (120:5mm);
\end{tikzpicture}}\rangle_\Sigma
+
\langle \mbox{\begin{tikzpicture}[yscale=-1,scale=0.24]
            \draw[line width=0.35mm,rotate=30] (30:5mm) to  (17:9mm);
            \draw[line width=0.35mm,rotate=30] (-30:5mm) to  (-17:9mm);
            \draw[line width=0.35mm,rotate=150] (30:5mm) to  (17:9mm);
            \draw[line width=0.35mm,rotate=150] (-30:5mm) to  (-17:9mm);
            \draw[line width=0.35mm,rotate=270] (30:5mm) to[out=0,in=0]  (-30:5mm);
            \draw[line width=0.35mm,rotate=270] [xshift=5mm](30:5mm) to[out=180,in=180]  (-30:5mm);
            \draw[line width=0.35mm,rotate=0] (60:5mm) to (120:5mm);
            \draw[line width=0.35mm,rotate=120] (60:5mm) to (120:5mm);
            \draw[line width=0.35mm,rotate=-120] (60:5mm) to (120:5mm);
\end{tikzpicture}}\rangle_\Sigma
+
\langle \mbox{\begin{tikzpicture}[yscale=-1,scale=0.24]
            \draw[line width=0.35mm,rotate=30] (30:5mm) to[out=0,in=0] (-30:5mm);
            \draw[line width=0.35mm,rotate=30] [xshift=5mm](30:5mm) to[out=180,in=180] (-30:5mm);
            \draw[line width=0.35mm,rotate=150] (30:5mm) to[out=0,in=0] (-30:5mm);
            \draw[line width=0.35mm,rotate=150] [xshift=5mm](30:5mm) to[out=180,in=180] (-30:5mm);
            \draw[line width=0.35mm,rotate=270] (30:5mm)  to  (17:9mm);
            \draw[line width=0.35mm,rotate=270] (-30:5mm) to  (-17:9mm);
            \draw[line width=0.35mm,rotate=0] (60:5mm) to (120:5mm);
            \draw[line width=0.35mm,rotate=120] (60:5mm) to (120:5mm);
            \draw[line width=0.35mm,rotate=-120] (60:5mm) to (120:5mm);
\end{tikzpicture}}\rangle_\Sigma
+
\langle \mbox{\begin{tikzpicture}[yscale=-1,scale=0.24]
            \draw[line width=0.35mm,rotate=30] (30:5mm) to[out=0,in=0] (-30:5mm);
            \draw[line width=0.35mm,rotate=30] [xshift=5mm](30:5mm) to[out=180,in=180] (-30:5mm);
            \draw[line width=0.35mm,rotate=150] (30:5mm) to[out=0,in=0]   (-30:5mm);
            \draw[line width=0.35mm,rotate=150] [xshift=5mm](30:5mm) to[out=180,in=180]  (-30:5mm);
            \draw[line width=0.35mm,rotate=270] (30:5mm) to[out=0,in=0]  (-30:5mm);
            \draw[line width=0.35mm,rotate=270] [xshift=5mm](30:5mm) to[out=180,in=180]  (-30:5mm);
            \draw[line width=0.35mm,rotate=0] (60:5mm) to (120:5mm);
            \draw[line width=0.35mm,rotate=120] (60:5mm) to (120:5mm);
            \draw[line width=0.35mm,rotate=-120] (60:5mm) to (120:5mm);
\end{tikzpicture}}\rangle_\Sigma
+
\langle \mbox{\begin{tikzpicture}[yscale=-1,scale=0.24,rotate=-120]
            \draw[line width=0.35mm,rotate=30] (30:5mm) to   (17:9mm);
            \draw[line width=0.35mm,rotate=30] (-30:5mm) to   (-17:9mm);
            \draw[line width=0.35mm,rotate=150] (30:5mm) to (17:9mm);
            \draw[line width=0.35mm,rotate=150] (-30:5mm) to (-17:9mm);
            \draw[line width=0.35mm,rotate=270] (30:5mm) to[out=0,in=0]  (-30:5mm);
            \draw[line width=0.35mm,rotate=270] [xshift=5mm](30:5mm) to[out=180,in=180]  (-30:5mm);
            \draw[line width=0.35mm,rotate=0] (60:5mm) to (120:5mm);
            \draw[line width=0.35mm,rotate=120] (60:5mm) to (120:5mm);
            \draw[line width=0.35mm, rotate=-120] (60:5mm) to (120:5mm);
\end{tikzpicture}}\rangle_\Sigma
+
\langle \mbox{\begin{tikzpicture}[yscale=-1,scale=0.24,rotate=-120]
            \draw[line width=0.35mm,rotate=30] (30:5mm) to[out=0,in=0] (-30:5mm);
            \draw[line width=0.35mm,rotate=30] [xshift=5mm](30:5mm) to[out=180,in=180] (-30:5mm);
            \draw[line width=0.35mm,rotate=150] (30:5mm) to[out=0,in=0] (-30:5mm);
            \draw[line width=0.35mm,rotate=150] [xshift=5mm](30:5mm) to[out=180,in=180] (-30:5mm);
            \draw[line width=0.35mm,rotate=270] (30:5mm)  to  (17:9mm);
            \draw[line width=0.35mm,rotate=270] (-30:5mm) to  (-17:9mm);
            \draw[line width=0.35mm,rotate=0] (60:5mm) to (120:5mm);
            \draw[line width=0.35mm,rotate=120] (60:5mm) to (120:5mm);
            \draw[line width=0.35mm,rotate=-120] (60:5mm) to (120:5mm);
\end{tikzpicture}}\rangle_\Sigma
+
\langle \mbox{\begin{tikzpicture}[yscale=-1,scale=0.24,rotate=120]
            \draw[line width=0.35mm,rotate=30] (30:5mm) to[out=0,in=0] (-30:5mm);
            \draw[line width=0.35mm,rotate=30] [xshift=5mm](30:5mm) to[out=180,in=180] (-30:5mm);
            \draw[line width=0.35mm,rotate=150] (30:5mm) to[out=0,in=0] (-30:5mm);
            \draw[line width=0.35mm,rotate=150] [xshift=5mm](30:5mm) to[out=180,in=180] (-30:5mm);
            \draw[line width=0.35mm,rotate=270] (30:5mm)  to (17:9mm);
            \draw[line width=0.35mm,rotate=270] (-30:5mm) to (-17:9mm);
            \draw[line width=0.35mm,rotate=0] (60:5mm) to (120:5mm);
            \draw[line width=0.35mm,rotate=120] (60:5mm) to (120:5mm);
            \draw[line width=0.35mm,rotate=-120] (60:5mm) to (120:5mm);
\end{tikzpicture}}\rangle_\Sigma
+\langle \mbox{\begin{tikzpicture}[scale=0.24,rotate=120]
            \draw[line width=0.35mm,rotate=30] (30:5mm) to   (17:9mm);
            \draw[line width=0.35mm,rotate=30] (-30:5mm) to   (-17:9mm);
            \draw[line width=0.35mm,rotate=150] (30:5mm) to  (17:9mm);
            \draw[line width=0.35mm,rotate=150] (-30:5mm) to  (-17:9mm);
            \draw[line width=0.35mm,rotate=270] (30:5mm) to[out=0,in=0]  (-30:5mm);
            \draw[line width=0.35mm,rotate=270] [xshift=5mm](30:5mm) to[out=180,in=180]  (-30:5mm);
            \draw[line width=0.35mm,rotate=0] (60:5mm) to (120:5mm);
            \draw[line width=0.35mm,rotate=120] (60:5mm) to (120:5mm);
            \draw[line width=0.35mm,rotate=-120] (60:5mm) to (120:5mm);
\end{tikzpicture}}\rangle_\Sigma
\\
&=&
\langle \mbox{\begin{tikzpicture}[yscale=-1,scale=0.20,rotate=0]
        \draw[line width=0.35mm,-,yshift=-0mm] (50:1.1cm) to (-130:1.1cm);
        \draw[line width=0.35mm,-] (0:1.2cm) to [out=140,in=40] (180:1.2cm);
        \draw[line width=0.35mm,-,yshift=-0mm] (-50:1.1cm) to (130:1.1cm);
\end{tikzpicture}} \rangle_\Sigma
\end{eqnarray*}

\end{proof}

\begin{proposition}
    Let $D,D'$ be two minimal genus diagrams of a flat knot $\alpha$. Then  $ J_D(z)=  J_{ D' }(z)$.
\end{proposition}
\begin{proof}
    By \cite{manturov}*{Theorem~3.2}, any two minimal genus diagrams are related by homotopy on the surface, namely,  the FR1,3-moves and the FR2-moves that do not change the genus.

We can check that the FR1 move changes the sign of the homological Kauffman bracket:

\begin{align*}
\langle \mbox{\begin{tikzpicture}[scale=0.7,rotate=90]
        \draw[very thick] (100:1cm) to [out=-90,in=90] ([yshift=-7mm]100:1cm) to [out=-90,in=180] (-45:4mm) to [out=0,in=-90] (0:5mm);
        \draw[very thick] (-100:1cm) to [out=90,in=-90] ([yshift=7mm]-100:1cm) to [out=90,in=180] (45:4mm) to [out=0,in=90] (0:5mm);
\end{tikzpicture}} \rangle_\Sigma
&=
\langle \mbox{\begin{tikzpicture}[scale=0.6,rotate=90]
        \draw[very thick] (0,-1cm) to (0,1cm);
        \draw[very thick] (0:8mm) arc (0:360:3mm);
\end{tikzpicture}}\rangle_\Sigma
+
\langle \mbox{\begin{tikzpicture}[yscale=0.24,xscale=0.6]
            \draw[very thick] (-1cm,0) to (-5mm,0) to [out=0,in=-50](-3mm,7mm) to [out=130,in=180] (0mm,13mm);
        \draw[very thick] (1cm,0) to (5mm,0) to [out=180,in=-130] (3mm,7mm) to [out=50,in=0] (0mm,13mm);
\end{tikzpicture}}\rangle_\Sigma
\\
&=
-2
\langle \mbox{
\begin{tikzpicture}[scale=0.7,rotate=90]
        \draw[very thick] (100:1cm) to [out=-90,in=90] ([yshift=-4mm]100:1cm) to [out=-90,in=90] (0:2mm);
        \draw[very thick] (-100:1cm) to [out=90,in=-90] ([yshift=4mm]-100:1cm) to [out=90,in=-90] (0:2mm);
\end{tikzpicture}}\rangle_\Sigma
+
\langle \mbox{
\begin{tikzpicture}[scale=0.7,rotate=90]
        \draw[very thick] (100:1cm) to [out=-90,in=90] ([yshift=-4mm]100:1cm) to [out=-90,in=90] (0:2mm);
        \draw[very thick] (-100:1cm) to [out=90,in=-90] ([yshift=4mm]-100:1cm) to [out=90,in=-90] (0:2mm);
\end{tikzpicture}}\rangle_\Sigma
\\
&=
-
\langle \mbox{
\begin{tikzpicture}[scale=0.7,rotate=90]
        \draw[very thick] (100:1cm) to [out=-90,in=90] ([yshift=-4mm]100:1cm) to [out=-90,in=90] (0:2mm);
        \draw[very thick] (-100:1cm) to [out=90,in=-90] ([yshift=4mm]-100:1cm) to [out=90,in=-90] (0:2mm);
\end{tikzpicture}}\rangle_\Sigma.
\end{align*}

Since the FR1-move also changes the number of crossings by one, we conclude that the flat Jones-Krushkal polynomial does not change under the FR1 move.

For the FR2 move, we can calculate the homological Kauffman bracket as below.

\begin{align*}
\langle \mbox{\begin{tikzpicture}[scale=0.24,rotate=-90]
        \draw[very thick] (120:1cm) to [out=-60,in=90] (0:2mm) to [out=-90, in=60] (240:1cm);
        \draw[very thick] (60:1cm) to [out=-120, in=90] (0:-2mm) to [out=-90, in=120] (-60:1cm);
\end{tikzpicture}} \rangle_\Sigma
&
=
\langle \mbox{\begin{tikzpicture}[scale=0.24,rotate=-90]
        \draw[very thick] (120:1cm) to [out=-60,in=60]    (240:1cm);
        \draw[very thick] (60:1cm) to [out=-120, in=120]  (-60:1cm);
\end{tikzpicture}}\rangle_\Sigma
+
\langle \mbox{\begin{tikzpicture}[scale=0.24,rotate=-90]
        \draw[very thick] (120:1cm) to [out=-60,in=90]    (-3mm,-2mm);
        \draw[very thick] (60:1cm) to [out=-120, in=90] (3mm,2mm) to (3mm,-2mm);
        \draw[very thick] (-3mm,-2mm) to[out=-90, in=-90]  (3mm,-2mm);
        \draw[very thick]  (250:1.2cm) to[out=90, in=90]  (-70:1.2cm);
\end{tikzpicture}}\rangle_\Sigma
+
\langle \mbox{\begin{tikzpicture}[scale=0.24,rotate=90]
        \draw[very thick] (120:1cm) to [out=-60,in=90]    (-3mm,-2mm);
        \draw[very thick] (60:1cm) to [out=-120, in=90] (3mm,2mm) to (3mm,-2mm);
        \draw[very thick] (-3mm,-2mm) to[out=-90, in=-90] (3mm,-2mm);
        \draw[very thick]  (250:1.2cm) to[out=90, in=90] (-70:1.2cm);
\end{tikzpicture}}\rangle_\Sigma
+
\langle \mbox{\begin{tikzpicture}[scale=0.24,rotate=-90]
        \draw[very thick] (-3mm,2mm) to[out=90, in=90]  (3mm,2mm);
        \draw[very thick]  (110:1.2cm) to[out=-90, in=-90]  (70:1.2cm);
        \draw[very thick]  (-3mm,2mm) to (-3mm,-2mm);
        \draw[very thick]  (3mm,2mm) to (3mm,-2mm);
        \draw[very thick] (-3mm,-2mm) to[out=-90, in=-90]  (3mm,-2mm);
        \draw[very thick]  (250:1.2cm) to[out=90, in=90]  (-70:1.2cm);
\end{tikzpicture}}\rangle_\Sigma
\\
&
=
\langle \mbox{\begin{tikzpicture}[scale=0.24,rotate=-90]
        \draw[very thick] (120:1cm) to [out=-60,in=60]    (240:1cm);
        \draw[very thick] (60:1cm) to [out=-120, in=120]  (-60:1cm);
\end{tikzpicture}}\rangle_\Sigma
+
\langle \mbox{\begin{tikzpicture}[scale=0.24,rotate=-90]
        \draw[very thick]  (110:1.2cm) to[out=-90, in=-90]  (70:1.2cm);
        \draw[very thick, yshift=9mm]  (250:1.2cm) to[out=90, in=90]  (-70:1.2cm);
\end{tikzpicture}}\rangle_\Sigma
+
\langle \mbox{\begin{tikzpicture}[scale=0.24,rotate=-90]
        \draw[very thick]  (110:1.2cm) to[out=-90, in=-90]  (70:1.2cm);
        \draw[very thick, yshift=9mm]  (250:1.2cm) to[out=90, in=90]  (-70:1.2cm);
\end{tikzpicture}}\rangle_\Sigma
-2
\langle \mbox{\begin{tikzpicture}[scale=0.24,rotate=-90]
        \draw[very thick]  (110:1.2cm) to[out=-90, in=-90]  (70:1.2cm);
        \draw[very thick, yshift=9mm]  (250:1.2cm) to[out=90, in=90]  (-70:1.2cm);
\end{tikzpicture}}\rangle_\Sigma
\\
&
=\langle \mbox{\begin{tikzpicture}[scale=0.24,rotate=-90]
        \draw[very thick] (120:1cm) to [out=-60,in=60]    (240:1cm);
        \draw[very thick] (60:1cm) to [out=-120, in=120]  (-60:1cm);
\end{tikzpicture}}\rangle_\Sigma.
\end{align*}

Therefore, the flat Jones-Krushkal polynomial is an invariant for minimal genus diagrams of flat knots.
\end{proof}

\begin{definition}
    The \emph{flat Jones-Krushkal polynomial} $J_\alpha(z)$ of flat knot $\alpha$ is defined as  $J_D(z)$ where $D$ is a minimal genus diagram of $\alpha$.
\end{definition}

The minimal genus is essential for $J_D(z)$ to be well-defined. For example, both the Gauss diagrams below represent the unknot.
But the left one has $J_D(z)=1$ and the right one has $J_{D'}(z)=2z+2$. 

\begin{figure}[ht]
   \centering
   \includegraphics[scale=1]{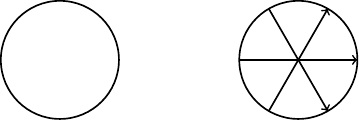}
   \caption{Two Gauss diagrams of the flat unknot}
   \label{fig:fk0}
\end{figure}

\begin{corollary}
    Let $D,D'$ be two minimal crossing Gauss diagrams of the same flat knot $\alpha$.
    Then  $J_D(z)= J_{ D' }(z)= J_\alpha(z)$.
\end{corollary}
\begin{proof}
    By \cite{manturov}*{Corollary~3.1}, a minimal crossing Gauss diagram of $\alpha$ achieves its minimal flat genus,
    and two minimal crossing diagrams $D,D'$ of the same flat knot are related by a finite sequence of FR3 moves.
    Therefore, their flat Jones-Krushkal polynomials satisfy
    $J_D(z)= J_{ D' }(z).$ 
\end{proof}

In \cite{flatknotinfo}, each flat knot type is represented by a minimal crossing Gauss diagram. One can directly use that Gauss diagram to calculate the flat Jones-Krushkal polynomial.

\subsection{Normalization and enhancement} \label{sec:norm-enh}
In this subsection, we present normalized and enhanced versions of the flat Jones-Krushkal polynomial. The enhancement is a stronger invariant, and it keeps track of the number of homologically nontrivial loops in each state. We also discuss the properties of these invariants.

\begin{definition}
The \emph{normalized flat Jones-Krushkal polynomial} of flat knot $\alpha$
diagram $D$ on a closed surface $\Sigma$
is defined by 
$$\ds \bar J_\alpha(z)=
\frac{J_D(z)}{ \varepsilon_D},
$$ 
where $D$ is a minimal diagram of $\alpha$,
$\varepsilon_D=-2$ if $[\omega_D(S^1)]=0 \in H_1(\Sigma;\Z/2)$, and   $\varepsilon_D=z$ otherwise.
\end{definition}

Then the flat Jones-Krushkal polynomial $\bar J_{3.1}(z)=-3z-5$.

\begin{proposition}
    For any given flat knot $\alpha$, the normalized flat Jones-Krushkal polynomial satisfies $\bar J_{\alpha}(-2)=1$.
\end{proposition}
\begin{proof}

    As mentioned
    in the discussion around Equation \eqref{eqn:flatarroweqn},
        assigning $a=1$  sends all Jones polynomials to $\sum_S (-2)^{|S|-1}=1$.
Then we have $$\bar J_{\alpha}(-2)=\sum_S (-2)^{|S|}/(-2)=\sum_S (-2)^{|S|-1}=1.$$
\end{proof}

Empirically, we noticed something curious about the roots of $ \bar J_\alpha(z)- 1$, which is that
    every almost classical flat knot up to 10 crossings satisfies:
    $$ \bar J_\alpha(-1)= 1. $$

\begin{conjecture}
    For any almost classical flat knot $\alpha$, the normalized flat Jones-Krushkal polynomial satisfies $\bar J_\alpha(-1)= 1. $
\end{conjecture}

If true, the condition provides a useful criterion for a flat knot to be almost classical. Among all the flat knots up to 7 crossings, only 6 satisfy the condition $\bar J_\alpha(-1)= 1$ but are not almost classical.

\begin{definition} 
The \emph{enhanced flat Jones-Krushkal polynomial} of a flat knot diagram $D$ realized as immersion on Carter surface of genus $g$ is defined by
$$\ds J^{en}_D(w,z)= (-1)^{cr(D)} \sum_{S\in \mathfrak{S} } \langle D\,|\,S \rangle_\Sigma \ w^{m(S)},$$ where
    $$ m(S) = |S|- \#\text{ of null-homologous curves in }S.$$
    \label{def:jkpoly}
\end{definition}

\begin{proposition}
    If $D,D'$ are two flat knot diagrams on a surface $\Sigma$ related by FR3 moves, then $  J^{en}_D(w,z)= J^{en}_{D'}(w,z)$.
\end{proposition}
\begin{proof}
    It is enough to verify the claim for two diagrams related by one FR3 move, and the details are similar to the calculation in Proposition~\ref{prop:jkfr3}.
\end{proof}

\begin{definition}
    The \emph{enhanced flat Jones-Krushkal polynomial} of a flat knot $\alpha$ is defined to be  $J^{en}_D(w,z)$, where $D$ is a minimal genus diagram of $\alpha$.
\end{definition}

Consider the flat knot 3.1, for example. One can compute that 
$$m(S)=\begin{cases} $2$ & \text{for $S \in \{000, 100, 011\},$ and} \\
$1$ & \text{for $ S \in \{001, 010, 101, 110, 111\}$.}
\end{cases}$$
Therefore, the enhanced Jones-Krushkal polynomial is given by
$$J^{en}_{3.1}(w,z)=-3w^2z^2-5wz.$$

\begin{proposition}
    Let $\alpha,\beta$ be two flat knots. If $J^{en}_\alpha (w,z)=J^{en}_\beta (w,z)$, then $J_\alpha (z)=J_\beta (z)$.
    \label{thm:enhanced}
\end{proposition}
\begin{proof}
    By definition, we have $\ds J^{en}_D(1,z)=J_D(z)$. Thus if two knots have the same enhanced flat Jones-Krushkal polynomials, then they must also have the same flat Jones-Krushkal polynomials.
\end{proof}
The converse of Proposition~\ref{thm:enhanced} is not true. For example, the flat knots 3.1 and 5.1 have $J_{3.1} (z)=J_{5.1} (z)=-3z^2-5z$, but their enhanced flat Jones-Krushkal polynomials are not equal:
\begin{align*}
 J^{en}_{3.1} &=-3 w^2 z^2-5wz , \\
 J^{en}_{5.1} &=(-4w^4  + 6 w^3  - 5w^2)z^2 - 5w z .
\end{align*}

\begin{proposition}
    If     $\alpha$ is checkerboard colorable, then $2|J^{en}_{\alpha}$.
    If     $\alpha$ is not checkerboard colorable, then $z|J^{en}_{\alpha}$ and $w|J^{en}_{\alpha}$.
\end{proposition}
\begin{proof}
    Let $\omega_\alpha(S^1)$ be the flat knot diagram on $\Sigma$.
    Then the sum of loops in each state $S$ of $\omega_\alpha(S^1)$ is equal to  $[ \omega_\alpha(S^1) ]\in H_1(\Sigma;\Z/2)$.
    Therefore, if  $\alpha$ is checkerboard colorable, then $k(S)>0$ for each state.
    If  $\alpha$ is not checkerboard colorable, then $m(S)\ge r(S)>0$ for each state.
\end{proof}

\begin{definition}
The \emph{normalized enhanced flat Jones-Krushkal polynomial} of a flat knot diagram $D$ realized as immersion on a Carter surface $\Sigma$ is defined by
$$\ds \bar J^{en}_\alpha(w,z)= 
\frac{J^{en}_D(z)}{ \varepsilon_D},$$ 
where $D$ is a minimal diagram of $\alpha$,
$\varepsilon_D=-2$ if $[\omega_D(S^1)]=0 \in H_1(\Sigma;\Z/2)$, and   $\varepsilon_D=z$ otherwise.
\end{definition}

Proposition~\ref{thm:enhanced} also holds for the normalized versions, and the same pair of flat knots (3.1 and 5.1) show that the converse is not true. Thus the enhanced flat Jones-Krushkal polynomial is a stronger invariant than the flat Jones-Krushkal polynomial.

The following example shows that the flat Jones-Krushkal polynomial is not multiplicative under connected sum. 
The flat knot 4.5 in Figure~\ref{fig:fk6132} is a connected sum of two diagrams of the unknot.
However, the flat Jones-Krushkal polynomial of 4.5 is $-4z+9$, which is nontrivial.

\subsection{Distinguishing flat knots (reprise)}
The flat Jones-Krushkal polynomial and its enhancement are powerful tools for distinguishing flat knots when the minimal flat genus is known.  
For example, the flat knots in Figure~\ref{fig:dup_7crossings_2} have 
$ \bar J_{7.46142}=-7z^2 - 21z - 13$ and $ \bar J_{7.46230}=-13z^2 - 39z - 25$, while other invariants such as the $u$-polynomial, flat arrow polynomial, $2$-strand  cabled flat arrow polynomial and based matrix of the pair are all identical.

\begin{figure}[ht]
  \centering
  \includegraphics[scale=1.1]{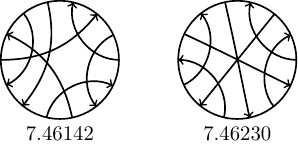}
  \caption[Flat knots not distinguished]{Two flat knots not distinguished by cabled arrow polynomials and based matrices}
   \label{fig:dup_7crossings_2}
\end{figure}

Notice that the two knots in Figure~\ref{fig:dup_7crossings_2} are almost classical.
We know that all the flat arrow polynomials and cabled flat arrow polynomials of almost classical knots are trivial. Additionally, there are almost classical knots with the trivial primitive based matrix. These flat knots are very difficult to separate from the unknot.
Recall from Example \ref{exa:ac8-19} that the flat knot ac8.19 in Figure~\ref{fig:fkac828} has trivial primitive based matrix. This flat knot is actually the first almost classical flat knot with trivial primitive based matrix. Previous invariants were unable to distinguish it from the flat unknot, but one can compute that it has flat Jones-Krushkal polynomial $ \bar J_{\rm{ac8.19}} = 24z^2 + 72z + 49$. Therefore, the flat Jones-Krushkal polynomial is able to distinguish it from the flat unknot.

There are a number of other almost classical flat knots having trivial primitive based matrix. For instance, there are several examples with $10$ crossings, but none with $9$ or fewer crossings. Each such $10$-crossing almost classical flat knot can be shown to have nontrivial flat Jones-Krushkal polynomial by direct computation, so they are all distinguished from the flat unknot using the flat Jones-Krushkal polynomial.

When combined with other flat knot invariants, the enhanced flat Jones-Krushkal polynomial distinguishes flat knots up to 6 crossings, leaving only 5 pairs of 7 crossing knots not separated, see Table~\ref{table:en_arr_phi_dis}. The five pairs of non-distinguished 7-crossing flat knots are depicted in Figure \ref{fig:fivepair}.

Up to 8 crossings, most of the non-distinguished flat knots are
composite. Indeed, restricting attention to \textit{prime} flat knots, up to 8 crossings, there is only one pair of flat knots with $7$ crossings that are not separated by the invariants, namely the two prime flat knots $7.21134$ and $7.32153$ in Figure \ref{fig:fivepair}.

\begin{table}[ht]
    \centering
\begin{tabular}{||c |c | c| c||} 
 \hline
 Crossings &  \# Flat knots  &   $A(\alpha^2)$, $C(\alpha^3)$, $\phi$  &$J^{en}_{\alpha}, A(\alpha^2), C(\alpha^3), \phi$\\ [0.5ex] 
 \hline\hline
3 & 1  &0&0\\
 \hline
4 & 11  &0&0\\
 \hline
5 & 120   &0&0\\
 \hline
6 & 2086  &0&0\\
 \hline
 7 & 46233  &12&10\\
 \hline
8 & 1241291 & 513 & 511 \\
 \hline 
\end{tabular} 
\vspace{2mm}
\caption[Distinguishing flat  knots]{Distinguishing flat  knots using $J^{en}_\alpha, A(\alpha^2), C(\alpha^3),$ and $\phi$-invariants}
\label{table:en_arr_phi_dis}
\end{table}

For checkerboard colorable knots, the invariants $J^{en}_{\alpha}, A(\alpha^2), \phi$ enable us to distinguish all checkerboard colorable flat knots up to 7 crossings, see Table \ref{table:en_arr_phi_dis2}.
There are four checkerboard colorable flat knots with 8 crossings with the same $J^{en}_{\alpha}, A(\alpha^2), \phi$, see Figure~\ref{fig:cc8}. Thus, there is a quadruple of 8-crossing
checkerboard colorable flat knots that are non-distinguished.

\begin{table}[ht]
    \centering
\begin{tabular}{||c |c | c| c||} 
 \hline
 Crossings &  \# Flat knots  &   $ A(\alpha^2),\phi$  &$J^{en}_{\alpha}, A(\alpha^2), \phi$\\ [0.5ex] 
 \hline\hline
4 & 1   &0&0\\
 \hline
5 & 5   &0&0\\
 \hline
6 & 33  &0&0\\
 \hline
7 & 347  &2 &0\\
 \hline
8 & 4451 & 5 & 4 \\
 \hline
\end{tabular}
\vspace{2mm}
\caption[Distinguishing checkerboard colorable flat  knots]{Distinguishing checkerboard colorable flat  knots using $J^{en}_{\alpha}, A(\alpha^2), \phi$}
\label{table:en_arr_phi_dis2}
\end{table}

As mentioned in the last section, the flat arrow polynomial is trivial for every almost classical flat knot. Thus, the only tools for distinguishing almost classical flat knots are the $\phi$-invariant and the enhanced flat Jones-Krushkal polynomial.
Using $J^{en}_\alpha,\phi$, we are able to distinguish almost classical knots up to $8$ crossings, see Table \ref{table:en_arr_phi_dis3}. The first almost classical flat knots that are not distinguished are the three pairs of flat knots with 9 crossings shown in Figure~\ref{fig:dup_8ac}.

\begin{table}[ht]
    \centering
\begin{tabular}{||c |c | c| c||} 
 \hline
 Crossings &  \# Flat knots  &   $\phi$  &$J^{en}_{\alpha}, \phi$\\ [0.5ex] 
 \hline\hline
5 & 1   &0&0\\
 \hline
6 & 1  &0&0\\
 \hline
 7 & 6  &2&0\\
 \hline
8 & 28 & 1 & 0 \\
 \hline
9 & 190 & 26 & 6 \\
 \hline
10 & 1682 & 175 & 39 \\
 \hline
\end{tabular}
\vspace{2mm}
\caption[Distinguishing almost classical flat  knots]{Distinguishing almost classical flat  knots using $J^{en}_\alpha,\phi$}
\label{table:en_arr_phi_dis3}
\end{table}

\begin{figure}[ht]
  \centering
  \includegraphics[scale=1.0]{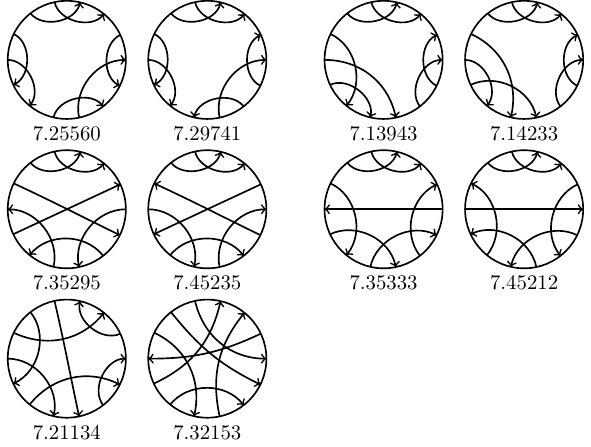}
\caption{Five pairs of flat knots not distinguished}
   \label{fig:fivepair}
\end{figure}

\begin{figure}[ht]
  \centering
  \includegraphics[scale=1.0]{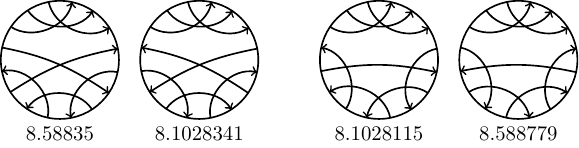}
  \caption{Four checkerboard colorable flat knots not distinguished}
   \label{fig:cc8}
\end{figure}

\begin{figure}[ht]
  \centering
  \includegraphics[scale=1.0]{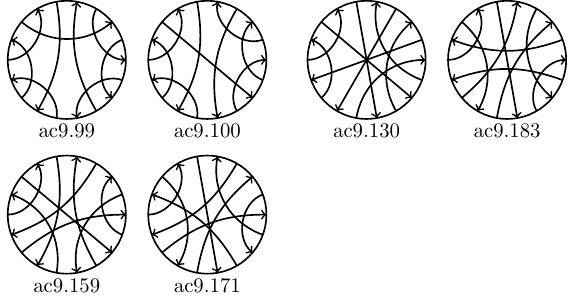}
\caption{Three pairs of almost classical flat knots that are not distinguished}
   \label{fig:dup_8ac}
\end{figure}

We conclude this section with a few open problems.
\begin{problem}
    Are there bigraded or triply graded invariants that categorify the flat Jones-Krushkal polynomial or its enhancement?  
    \label{prob:cate_jk}
\end{problem}

\begin{problem}
    Can one use the flat Jones-Krushkal polynomial or its enhancement to extract slice obstructions for flat knots? 
    \label{prob:slice_jk}
\end{problem}

\begin{problem}
    Which polynomials can be realized as flat Jones-Krushkal  polynomials of flat knots? Which polynomials can be realized as enhanced flat Jones-Krushkal  polynomials of flat knots? 
    \label{prob:rel_jk}
\end{problem}

\begin{problem}
    Referring to \cite{miller22}, does the \emph{homological arrow polynomial} lead to an invariant of flat knots? How powerful is it? 
    \label{prob:arr_jk}
\end{problem}

\begin{problem}
    Find a good algorithm for computing the colored flat Jones-Krushkal  and enhanced  flat Jones-Krushkal polynomials. How powerful are they in terms of classifying flat knots? Specifically, to what extent are the 1,2, and 3 strand cabled versions of the flat Jones-Krushkal polynomials classifying among low-crossing flat knots? 
    
\end{problem}

\section{Conclusion}

As a result, we are able to tabulate the first 1,289,741 flat knots.
We completely distinguish flat knots up to $6$ crossings, and our method works for flat knots with $7$ crossings leaving only five pairs of ambiguities. 
For flat knots with $8$ crossings, the method leaves a total of 511 undistinguished, but most of the ambiguities arise from composite flat knots. Indeed, if we restrict our attention to prime flat knots, the invariants separate all flat knots up to $8$ crossings except for one pair with $7$ crossings, namely
the last pair of flat knots $7.21134, 7.32153$ in Figure~\ref{fig:fivepair}.

\begin{table}[ht]
    \centering
\begin{tabular}{||c |c |c| c||} 
 \hline
 Crossings &  \# Flat knots & \# Checkerboard colorable & \# Almost classical\\ [0.5ex] 
 \hline\hline
3 & 1 & 0 & 0\\
 \hline
4 & 11 & 1 & 0\\
 \hline
5 & 120 & 5 & 1\\
 \hline
6 & 2086 & 33 & 1\\
 \hline
7 & 46233 & 347 & 6\\
 \hline
8 & 1241291 & 4451 & 28\\
 \hline
9 & & 71404 &190\\
 \hline
10 & & 1303643 & 1682\\
 \hline
11 & &  & 18002\\
  \hline
12 & &  & 220849\\
\hline
\end{tabular}
\vspace{2mm}
\caption{The number of tabulated flat knots}
\label{table:number}
\end{table}

We also consider the classification problem for the subclasses of checkerboard colorable and almost classical flat knots.
For instance, we tabulate the first 1,379,884 checkerboard colorable flat knots and the first 240,759 almost classical flat knots.
For checkerboard colorable flat knots, the invariants completely distinguish them up to $7$ crossings and the method works for checkerboard colorable flat knots with 8 crossings leaving only four ambiguities.
For almost classical flat knots, the invariants completely distinguish them up to $8$ crossings, leaving only $3$ pairs of $9$-crossing knots undistinguished.

\begin{figure}[ht]
   \centering
   \includegraphics[scale=1.2]{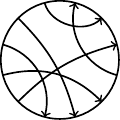}
   \caption{Flat knot 6.540}
   \label{fig:fk6540}
\end{figure}
Furthermore, there are only one flat knot with 
6 crossings and eight flat knots with 7 crossings whose slice status remains unknown. They are displayed in Figure~\ref{fig:fk_7slicesuspect}.
This is the result of filtering out all flat knots that are not algebraically slice, as well as any flat knot that parity projects to a non-slice flat knot. Then we search for saddle moves using the fillings. This approach was successful in slicing many of the remaining flat knots, and Figure~\ref{fig:fk_7slicesuspect}  shows the residual set of eight flat knots where this method was inconclusive. Each of these 7-crossing knots is algebraically slice, but none of them are known to be slice.
\begin{figure}[ht]
   \centering
   \includegraphics[width=0.9\textwidth]{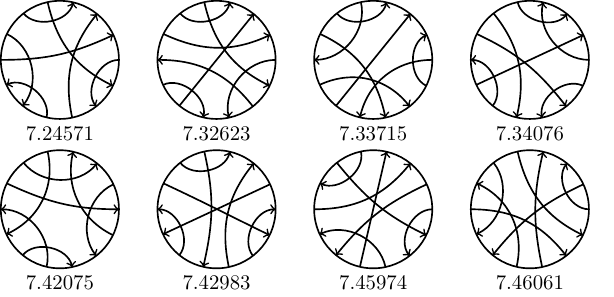}
   \caption{Eight flat knots of unknown slice status}
   \label{fig:fk_7slicesuspect}
\end{figure}

\section*{Acknowledgements}
The main results in this paper are derived from the author's Ph.D. thesis \cite{chen-thesis} written under the supervision of Hans Boden at McMaster University.

\Addresses
\end{document}